\documentclass{article}
\usepackage[margin=1in]{geometry}
\usepackage{amsmath}
\usepackage{amssymb}
\usepackage{amsthm}
\usepackage{amsfonts}
\usepackage{mathrsfs}
\usepackage{mathtools}
\usepackage{graphicx}
\usepackage{caption}
\usepackage{setspace}
\usepackage{amsthm}
\usepackage[colorlinks=true, allcolors=blue]{hyperref}
\usepackage{fullpage}
\usepackage{comment}
\usepackage{xcolor}
\usepackage{enumerate}
\usepackage[capitalize]{cleveref}
\usepackage{thmtools}
\usepackage{bbm}

\usepackage{tikz}

\newcommand{\Z}{\mathbb Z}

\newcommand{\Q}{\mathbb Q}
\newcommand{\ceil}[1]{\left\lceil #1 \right\rceil}
\newcommand{\floor}[1]{\left\lfloor #1 \right\rfloor}
\newcommand{\sm}{\setminus}
\renewcommand{\P}{\mathcal P}
\newcommand{\A}{\mathcal A}

\newcommand{\wsat}{\operatorname{wsat}}
\newcommand{\wSAT}{\operatorname{wSAT}}
\newcommand{\uwSAT}{\underline{\operatorname{wSAT}}}

\newtheorem{theorem}{Theorem}[section]
\newtheorem{lemma}[theorem]{Lemma}

\newtheorem{prop}[theorem]{Proposition}

\newtheorem{claim}[theorem]{Claim}
\newtheorem*{claim*}{Claim} 

\newtheorem{conj}[theorem]{Conjecture}

\theoremstyle{definition}
\newtheorem{definition}[theorem]{Definition}
\newtheorem{remark}[theorem]{Remark}

\begin{document}
\title{Rational values of the weak saturation limit}
\author{Ruben Ascoli\thanks{School of Mathematics, Georgia Institute of Technology, Atlanta, GA 30332. Email: rascoli3@gatech.edu}\ \  and Xiaoyu He\thanks{School of Mathematics, Georgia Institute of Technology, Atlanta, GA 30332. Email: xhe399@gatech.edu.}}
\date{\today}
\maketitle

\begin{abstract}
Given a graph $F$, a graph $G$ is weakly $F$-saturated if all non-edges of $G$ can be added in some order so that each new edge introduces a copy of $F$. The weak saturation number $\wsat(n, F)$ is the minimum number of edges in a weakly $F$-saturated graph on $n$ vertices. Bollob\'as initiated the study of weak saturation in 1968 to study percolation processes, which originated in biology and have applications in  physics and computer science. It was shown by Alon that for each $F$, there is a constant $w_F$ such that $\wsat(n, F) = w_Fn + o(n)$. We characterize all possible rational values of $w_F$, proving in particular that $w_F$ can equal any rational number at least $\frac 32$.
\end{abstract}

\section{Introduction}

Given a (finite, simple) graph $F$, an \emph{$F$-bootstrap percolation process} on a graph $G$ is a sequence $G = G_0 \subseteq G_1 \subseteq \ldots \subseteq G_m$ where $G_{i}$ is obtained from $G_{i-1}$ by adding an edge which creates a new copy of $F$. Bollob\'as initiated the study of $F$-bootstrap percolation processes in 1968 \cite{Bol68} as a graph-theoretic analogue of self-reproducing cellular automata from biology \cite{vn} and percolation models prevalent in physics, see e.g. \cite{adler}. Monotone cellular automata such as bootstrap percolation have received significant attention in both mathematics and statistical physics with the resolution of the so-called universality conjecture \cite{BBMS}.

A graph $G$ on $n$ vertices is \emph{weakly $F$-saturated} 
if there exists an $F$-bootstrap percolation process $G = G_0 \subseteq G_1 \subseteq \ldots \subseteq G_m = K_n$. For given $F$ and $n \ge 1$, the \emph{weak saturation number} $\wsat(n, F)$ is the minimum number of edges in a weakly $F$-saturated graph on $n$ vertices.
We denote by $\wSAT(n,F)$ the set of all weakly $F$-saturated graphs on $n$ vertices, and by $\uwSAT(n,F)$ the subset of $\wSAT(n,F)$ of \emph{minimum} weakly $F$-saturated graphs, i.e. those with exactly $\wsat(n,F)$ edges.

Determining the exact value of $\wsat(n, F)$ proves to be a difficult problem in general. Even the value of $\wsat(n, K_s)$, which is the original problem posed by Bollob\'as, was only determined by Frankl~\cite{Fra} in 1982 by modifying an argument of Lov\'asz~\cite{Lov}; two further proofs followed by Kalai~\cite{Kal84, Kal85} and another by Alon~\cite{Alon}. Specifically, they showed that $\wsat(n, K_s) = \binom{n}{2}-\binom{n-s+2}{2} = (s-2)n - \binom{s-1}{2}$.  The last decade has seen much progress in determining exact weak saturation numbers for various classes of graphs, including stars, complete bipartite graphs, and other small graphs; see for example \cite{FG14, FGJ, PC16, KMM, XW23}. 
Recent work \cite{ST23, tz24, t25} has also studied the weak saturation number for hypergraphs.
However, exact values of $\wsat(n,F)$ are still only known for very specific graphs. 

As for the asymptotics of $\wsat(n,F)$, Alon~\cite{Alon} showed via a subadditivity argument that the limit $$w_F :=\lim_{n\to\infty}\dfrac{\wsat(n,F)}{n},$$
which we call the \textit{weak saturation limit} of $F$, always exists. 
Faudree, Gould, and Jacobson~\cite{FGJ} showed that $w_F$ is closely controlled by the minimum degree $\delta_F$ of $F$, via the inequalities
\begin{equation}\label{eq:delta-bound}
\frac{\delta_F}{2}-\frac{1}{\delta_F+1}\leq w_F  \leq \delta_F-1.
\end{equation}
Recently, Terekhov and Zhukovskii~\cite{tz23} discovered an error in the proof of the lower bound in \eqref{eq:delta-bound} and gave a new proof. They also asked how the values of $w_F$ are distributed within the interval~\eqref{eq:delta-bound}, and constructed explicit examples showing that $w_F$ can take on any multiple of $\frac{1}{\delta_F+1}$ within the interval.

The main purpose of this paper is to fully characterize the possible rational values of $w_F$.

\begin{theorem}\label{thm:main}
The rational values of the weak saturation limit $w_F$ are precisely 
$$\{0\} \cup \{1-1/k: k\geq 3\}\cup \{1\}\cup\{3/2 - 1/k: k\geq 4 \text{ even}\} \cup \big(\Q\cap [3/2, \infty)\big).$$
\end{theorem}

Surprisingly, any rational number at least $3/2$ is attainable by $w_F$, while below $3/2$, the situation is more discrete, with only a single accumulation point at $1$.
\cref{thm:main} breaks naturally into two cases, which together answer the question of Terekhov and Zhukovkii for rational values. 

\begin{theorem}[Sparse Regime]\label{thm:sparse} Suppose $\delta \ge 2$ and $w < \delta / 2$. There exists a graph $F$ with $\delta_F = \delta$ and $w_F = w$ if and only if $w = \delta/2 - 1/k$ where $k$ is the number of vertices in some $\delta$-regular graph, i.e. $k\ge \delta + 1$ and $k$ must be even if $\delta$ is odd.
\end{theorem}

For $w \ge \delta/2$, the behavior is radically different.

\begin{theorem}[Dense Regime]\label{thm:dense}
For any integer $\delta \geq 2$ and any rational $w\in[\delta/2, \delta-1]$, there exists a graph $F$ with $\delta_F = \delta$ and $w_F = w$.
\end{theorem}

By \eqref{eq:delta-bound}, we have $w_F = 0$ if $\delta = 1$, so \cref{thm:main} follows immediately from \cref{thm:sparse} and \cref{thm:dense}. 
Roughly speaking, \cref{thm:sparse} and \cref{thm:dense} together prove the existence of extremal bootstrap percolation processes of every possible ``rational" or ``periodic" behavior, and \cref{thm:sparse} additionally disallows the possibility of ``aperiodic" behaviors in the sparse regime. It would be extremely interesting to prove that all extremal processes are periodic in the dense regime as well; see \cref{conclusion} for more discussion of this problem.

Alon's~\cite{Alon} subadditivity argument showed that for any graph $F$, we have $\wsat(n,F) = w_Fn + o(n)$. In 1992, Tuza conjectured that the lower order term is in fact $O(1)$, which is of the same spirit as the conjecture that all extremal processes are periodic.
\begin{conj}[\cite{Tuza}] \label{conj:tuza} For any graph $F$, there exists $c_F$ such that $w_Fn - c_F \leq \wsat(n, F) \leq w_Fn + c_F$ for all large enough $n$.
\end{conj}
The lower inequality is easy to show; see \cite{tz23} for a short proof, as well as a discussion of for which classes of graphs the conjecture is known. 
We establish Tuza's conjecture when $w_F$ in the ``sparse regime."
\begin{theorem}\label{thm:sparse tuza}
    For any graph $F$, if $w_F < \delta_F/2$, then $\wsat(n,F) = w_Fn + O(1)$.
\end{theorem}
This paper is organized as follows. In \cref{sec:prelim}, we prove our main technical lemma, inspired by~\cite{tz23}, for constructing $F$ with a target value of $w_F$. In \cref{sparse section}, we prove \cref{thm:sparse}, where the main difficulty is showing that no small values of $w_F$ are possible other than the specific family described. \cref{thm:sparse tuza} follows as an easy consequence of that proof. In \cref{dense section}, we prove \cref{thm:dense} by showing that any rational value in the target interval can be achieved; the construction is obtained by carefully modifying a Cayley graph and attaching it to a large clique. We conclude the paper in \cref{conclusion} with a few of the fascinating open problems that remain in this area.

\section{Preliminaries} \label{sec:prelim}
In this section we prove our main technical tool for constructing $F$ with a given weak saturation limit. Throughout, we only consider graphs $F$ with $\delta_F\geq 2$. We begin with the following definitions, slightly modified from \cite{tz23}.

\begin{definition}
For a subset $S$ of $V(F)$, let $m_F(S)$ denote the number of edges in $F$ with at least one end in $S$. Given a nonempty subset $S\subseteq V(F)$, define $\gamma_F(S) = (m_F(S)-1)/|S|$, and let $\gamma_F = \min_{\emptyset\subsetneq S\subseteq V(F)} \gamma_F(S)$.
\end{definition}

Terekhov and Zhukovskii related $w_F$ to $\gamma_F$ as follows.
\begin{lemma}[{\cite[Theorem 2.1 and Section 3.1]{tz23}}]
\label{lem:tzbound}
For any graph $F$, $w_F \ge \gamma_F$.
\end{lemma}

Although $w_F$ and $\gamma_F$ are not always equal, our key lemma states that every value attained by $\gamma$ is attained by $w$.

\begin{lemma}\label{lem:f-tilde}
Given a graph $F$, there exists a graph $\tilde F$ with $\delta_{\tilde F} = \delta_F$ and $w_{\tilde F} = \gamma_F$.
\end{lemma}
\begin{proof}
Let $\emptyset \subsetneq S\subseteq V(F)$ be a vertex subset achieving $\gamma_F = \gamma_F(S)$. We may assume that there exists $f_0\in V(F)\sm S$ which is not adjacent to any vertex in $S$. To accomplish this, we can for example modify $F$ by taking its disjoint union with a very large clique.

Let $\tilde F$ be the disjoint union of all spanning supergraphs of $F$, including $F$ itself. Clearly, $\delta_{\tilde F} = \delta_F$. It remains to show that $w_{\tilde F} = \gamma_F$, which we break into two claims.

\begin{claim*}
    With $\tilde F$ as above, $w_{\tilde F} \geq \gamma_F$.
\end{claim*}
\begin{proof}
By \cref{lem:tzbound}, it suffices to show that $\gamma_{\tilde F} = \gamma_F$. Let $\emptyset\subsetneq T \subseteq V(\tilde F)$ be a vertex subset achieving $\gamma_{\tilde F} = \gamma_{\tilde F}(T)$. We first reduce to the case that $T$ lies in a single spanning supergraph of $F$. Suppose $T$ intersects multiple spanning supergraphs of $F$ in $\tilde{F}$, and call the intersections $T_1, \ldots, T_r$. Recalling the formula $\gamma_{\tilde F}(T) = (m_{\tilde F}(T)-1)/|T|$, observe that if we restrict to the $T_i$ which minimizes $m_{\tilde F}(T_i)-1)/|T_i|$, then $\gamma_{\tilde F}(T_i) \leq \gamma_{\tilde F}(T)$. 

Now, if $T\subseteq V(\tilde F)$ lies inside one of the supergraphs of $F$ in $\tilde F$, then to minimize $\gamma_{\tilde F}(T)$, it is optimal to choose the corresponding subset inside $F$ itself since it has the fewest edges; but then, $\gamma_{\tilde F}(T) = \gamma_{F}(S)$ 
by definition of $S$.
Thus $\gamma_{\tilde F} \geq \gamma_F$. Conversely, making precisely the choice of the vertices $S$ inside the copy of $F$ in $\tilde F$, we obtain $\gamma_{\tilde F} \leq \gamma_F$, concluding the proof of the claim.
\end{proof}

\begin{claim*} 
With $\tilde F$ as above, $w_{\tilde F} \leq \gamma_F$.
\end{claim*}
\begin{proof}
Define weakly $\tilde F$-saturated graphs $G_i$ for $i\geq 0$ as follows. Let $K$ be a clique of size $|V(\tilde F)|+1$, and choose an arbitrary set $U\subseteq V(K)$ of size $|V(F)\sm S|$. Fix an arbitrary edge $e^*\in E(F)$ incident to $S$. Let $V(G_i) = K \sqcup S_1 \sqcup \cdots \sqcup S_i$ where each $S_j$ is a new vertex set of size $|S|$. The edges of $G_i$ are such that $K$ forms a clique, and for each $1\le j \le i$, the induced subgraph $G_i[U\cup S_j]$ is an isomorphic copy of $F-e^*$, where the isomorphism identifies the vertex set $S_j$ with $S\subseteq V(F)$, except that the subgraph on $U$ is replaced by a clique. There are no edges between two different $S_j$.

Clearly, $\lim_{i\to\infty} \frac{e(G_i)}{|V(G_i)|} = \gamma_F(S) = \gamma_F$. We argue that $G_i$ is weakly $\tilde F$-saturated by induction on $i$. The base case $i=0$ is trivial since $G_i$ is a clique of size $|F(\tilde F)|+1$.
Take $i\geq 1$ and assume $G_{i-1}$ is weakly $\tilde F$-saturated. Using this assumption, we now describe an $F$-bootstrap percolation process that restores all the edges missing from $G_i$. First, since $G_i \sm S_i$ is isomorphic to $G_{i-1}$, by induction we may restore all the missing edges in $G_i$ except for those incident to $S_i$. Let $K'=V(G_i) \sm S_i$, which at this point forms a clique. 

Next, restore the image of the deleted edge $e^*$ incident to $S_i$. This creates a new copy of $\tilde{F}$ since a new copy of $F$ is formed by $e^*$, and the rest of $\tilde F$ is found in $K'$. Next, restore the remaining edges within $S_i$ or with one end in $S_i$ and the other in $U$. Each such edge creates a new copy of $\tilde F$ since a new copy of some supergraph of $F$ is introduced on the vertex set $S_i\cup U$, and the rest of $\tilde F$ can be found in $K'$.     

Finally, restore all remaining edges between $S_i$ and $K'\sm U$. For each such edge $uv$ with $u\in S_i$ and $v\in K'\sm U$, consider the graph $F+u'f_0$, where $u'\in S\subseteq V(F)$ corresponds to $u\in S_i$ and (recall) $f_0$ is the vertex in $V(F)\sm S$ not adjacent to any vertex in $S$. In our bootstrap percolation process on $G_i$, the edge $uv$ creates a new copy of $F+u'f_0$ on the vertex set $S_i\cup \{v\} \cup U \sm \{f_0'\}$, where $f_0'$ denotes the image of $f_0$ in $U$. 

Thus, each of these remaining edges creates a new copy of $\tilde F$, with its other components found in $K'$ as usual.
\end{proof}

Combining the two claims, the lemma is proved.
\end{proof}
In the above construction, we are only concerned with the leading order of $e(G_i)$ in $n$, so we were relatively wasteful about the lower-order terms. We discuss the lower order term in $\wsat(n,F)$ in greater detail in \cref{conclusion}.

\section{The Sparse Regime}\label{sparse section}

In this section we prove \cref{thm:sparse}. We begin by showing all the values described in the theorem are achievable.

\begin{lemma}
    For any $\delta \ge 2$, if $F$ is a $2$-edge-connected $\delta$-regular graph, then $\gamma_F = \delta / 2 - 1/|V(F)|$.
\end{lemma}
\begin{proof}
    For every $\emptyset\subsetneq S \subsetneq V(F)$, since $F$ is $2$-edge-connected, $S$ must have at least two edges to $V(F)\sm S$, so $m_F(S) \ge \delta |S|/2 + 1$. Thus, $\gamma_F(S) \ge \delta/2 > \gamma_F(V(F)) = \delta / 2 - 1/|V(F)|$, as desired.
\end{proof}
Observe that if a $\delta$-regular graph exists on $k$ vertices, then a $2$-edge-connected one exists as well, since we can always force the graph to have a Hamilton cycle. The if direction of \cref{thm:sparse} thus follows from the above lemma and \cref{lem:f-tilde}. It remains to show the only if direction, which states that the only possible values of $w_F$ smaller than $\delta/2$ are those achieved above.

The intuition for this proof is that when $w_F$ is less than $\delta/2$, essentially the only way to form a large minimum weakly $F$-saturated graph is to repeatedly copy a sparse connected component. Such a component may not always be present in an arbitrary member $G$ of $\uwSAT(n,F)$, but as we show below, there is a ``rotation" operation we can apply to any such $G$ to obtain a $G'\in \uwSAT(n,F)$ with a sparse component. 

In the remainder of this section, $F$ is a fixed graph of minimum degree $\delta$ with $w_F < \delta/2$, and $G\in \uwSAT(n_0, F)$ is a minimum weakly $F$-saturated graph on $n_0$ vertices, where $n_0 \gg |V(F)|$ is very large (so that $e(G)/|V(G)|$ is very close to $w_F$).
\begin{definition}
Fix an $F$-bootstrap percolation process on $G$. Let all vertices of $G$ start out \emph{inactive}. When an edge $e\not\in E(G)$ is restored in the process, choose a new copy of $F$ formed using $e$ (ignoring any other new copies of $F$ if there are more) and let $a(e)\subseteq V(G)$ be the set of inactive vertices used in that new copy. Those vertices are \emph{activated} by $e$ and are considered \emph{active} for the rest of the process. We call $e$ the \emph{activating edge} of $a(e)$, and we call edges that activate a nonzero number of vertices \emph{activating edges}.

Observe that the sets $\{a(e): e\not\in E(G), a(e)\neq \emptyset\}$ partition $V(G)$. (Indeed, if a vertex $u$ is never activated, then $u$ is complete to the rest of $G$ and is never required for new copies of $F$ in the percolation process, contradicting that $G\in \uwSAT(n_0,F)$ for $n_0 \geq |V(F)|$.) Call this partition the \emph{activation partition} of $G$, and denote it by $\A$. Say that a part $a(e)\in \A$ \emph{owns} edges of $G$ with at least one end in $a(e)$ that are used in the new copy of $F$ formed by $e$. Edges of $G$ not owned by any $a\in \A$ are called \emph{free}.
\end{definition}

The definition of ownership above has several important consequences.
\begin{prop}\label{prop:ownership} For each edge $e\in E(G)$, the following hold.
\begin{enumerate} 
\item If there is a part $a\in \A$ which owns $e$, then $a$ contains at least one end of $e$. 
\item At most one part in $\A$ owns $e$. 
\end{enumerate}
\end{prop}
\begin{proof}
The first statement follows directly from the definition. For the second statement, suppose that an edge $e'\in E(G)$ has one end in each of $a(e_1)$ and $a(e_2)$, where $e_1$ is restored before $e_2$ in the percolation process. We claim that $e'$ cannot be owned by $a(e_1)$. Indeed, if $a(e_1)$ owns $e'$, then $e'$ is used in the new copy of $F$ formed by $e_1$, so its end in $a(e_2)$ is activated by $e_1$ and lies in $a(e_1)$ as well, a contradiction. 
\end{proof}

\begin{definition}[$\A$-matching, rotation]
Let $\hat G$ denote the graph obtained from $G$ by adding the activating edge of each part in $\A$. The graph $\hat G$ has the same vertex set, activation partition, and ownership of edges as $G$. Additionally, for $\hat G$, if an activating edge has at least one end in the part it activates, then it is owned by that part; otherwise, it is free. 

Define an \textit{$\A$-matching} to be a subset $M\subseteq E(\hat G)$ consisting of exactly one edge owned by each part of $\A$. 
A graph $G'$ is called a \emph{rotation} of $G$ if it is obtained from $\hat G$ by removing an $\A$-matching. 
\end{definition}

We remark that 
\cref{prop:ownership} holds for $\hat G$ as well as for $G$.

\begin{lemma}\label{lem:rotation}
Let $G'$ be a rotation of $G\in\uwSAT(n_0, F)$. Then, $G'\in \uwSAT(n_0, F)$.
\end{lemma}
\begin{proof}
The graphs $G$, $\hat G$, and $G'$ have the same vertex set. 
In the $F$-bootstrap percolation process on $G$, say that the parts $a_1, \ldots, a_k$ of $\A$ were activated in that order by edges $e_1, \ldots, e_k$. Label the edges of the $\A$-matching removed from $\hat G$ to form $G'$ as $e_1', \ldots, e_k'$ where part $a_i$ owns edge $e_i'$. Define an $F$-bootstrap percolation process on $G'$ 
by restoring edges in the same order as in $G$, but when an edge $e_i$ is restored in $G$, instead restore $e_i'$ in $G'$.

We see that this is a valid percolation process from the following observation, which can be proved by induction running the percolation processes for $G$ and $G'$ in synchrony. At every point in time, if $A$ is the current set of activated vertices, then the induced subgraphs $G[A]$ and $G'[A]$ are the same.

Thus $G' \in \wsat(n_0,F)$, and the fact that $G'\in \uwSAT(n_0,F)$ follows from $e(G') = e(G)$.
\end{proof}

See Figure \ref{fig:rotation example} for an example of the rotation procedure.

\tikzset{every picture/.style={line width=0.75pt}}
\begin{figure}[!htb]
\minipage[t]{0.25\textwidth}
\centering 
\begin{tikzpicture}[x=0.75pt,y=0.75pt,yscale=-1.3,xscale=1.3]
\draw   (70.27,50.13) -- (100.27,50.13) -- (121.47,69.73) -- (100.27,90.93) -- (70.67,90.53) -- (49.47,70.13) -- cycle ;
\draw    (70.27,50.13) -- (70.67,90.53) ;
\draw    (49.47,70.13) -- (100.27,90.93) ;
\draw    (49.47,70.13) -- (100.27,50.13) ;
\draw    (49.47,70.13) -- (121.47,69.73) ;
\draw    (70.27,50.13) -- (121.47,69.73) ;
\draw    (70.27,50.13) -- (100.27,90.93) ;
\draw    (100.27,50.13) -- (100.27,90.93) ;
\draw    (100.27,50.13) -- (70.67,90.53) ;
\draw    (121.47,69.73) -- (70.67,90.53) ;
\draw   (79.36,126.65) -- (57.43,137.57) -- (35.44,126.71) -- (43.78,109.08) -- (70.92,109.05) -- cycle ;
\draw    (35.45,126.71) -- (79.36,126.65) ;
\draw    (43.78,109.08) -- (57.44,137.57) ;
\draw    (70.67,90.53) -- (70.92,109.05) ;
\draw   (141.2,126.59) -- (116.24,137.19) -- (91.2,126.66) -- (100.69,109.54) -- (131.59,109.5) -- cycle ;
\draw    (91.2,126.66) -- (141.2,126.59) ;
\draw    (100.69,109.54) -- (116.24,137.19) ;
\draw    (100.27,92.17) -- (100.69,110.3) ;
\draw    (131.59,109.5) -- (116.24,137.19) ;
\draw  [fill={rgb, 255:red, 0; green, 0; blue, 0 }  ,fill opacity=1 ] (130.29,109.5) .. controls (130.29,109.09) and (130.87,108.75) .. (131.59,108.75) .. controls (132.31,108.75) and (132.89,109.09) .. (132.89,109.5) .. controls (132.89,109.92) and (132.31,110.26) .. (131.59,110.26) .. controls (130.87,110.26) and (130.29,109.92) .. (130.29,109.5) -- cycle ;
\draw  [fill={rgb, 255:red, 0; green, 0; blue, 0 }  ,fill opacity=1 ] (139.9,126.59) .. controls (139.9,126.18) and (140.48,125.84) .. (141.2,125.84) .. controls (141.92,125.84) and (142.5,126.18) .. (142.5,126.59) .. controls (142.5,127.01) and (141.92,127.35) .. (141.2,127.35) .. controls (140.48,127.35) and (139.9,127.01) .. (139.9,126.59) -- cycle ;
\draw  [fill={rgb, 255:red, 0; green, 0; blue, 0 }  ,fill opacity=1 ] (114.94,137.19) .. controls (114.94,136.78) and (115.52,136.44) .. (116.24,136.44) .. controls (116.96,136.44) and (117.54,136.78) .. (117.54,137.19) .. controls (117.54,137.61) and (116.96,137.95) .. (116.24,137.95) .. controls (115.52,137.95) and (114.94,137.61) .. (114.94,137.19) -- cycle ;
\draw  [fill={rgb, 255:red, 0; green, 0; blue, 0 }  ,fill opacity=1 ] (89.9,126.66) .. controls (89.9,126.24) and (90.49,125.9) .. (91.2,125.9) .. controls (91.92,125.9) and (92.5,126.24) .. (92.5,126.66) .. controls (92.5,127.07) and (91.92,127.41) .. (91.2,127.41) .. controls (90.49,127.41) and (89.9,127.07) .. (89.9,126.66) -- cycle ;
\draw  [fill={rgb, 255:red, 0; green, 0; blue, 0 }  ,fill opacity=1 ] (99.39,109.54) .. controls (99.39,109.12) and (99.97,108.79) .. (100.69,108.79) .. controls (101.41,108.79) and (101.99,109.12) .. (101.99,109.54) .. controls (101.99,109.96) and (101.41,110.3) .. (100.69,110.3) .. controls (99.97,110.3) and (99.39,109.96) .. (99.39,109.54) -- cycle ;
\draw  [fill={rgb, 255:red, 0; green, 0; blue, 0 }  ,fill opacity=1 ] (42.64,109.08) .. controls (42.64,108.65) and (43.15,108.3) .. (43.78,108.3) .. controls (44.41,108.3) and (44.92,108.65) .. (44.92,109.08) .. controls (44.92,109.51) and (44.41,109.86) .. (43.78,109.86) .. controls (43.15,109.86) and (42.64,109.51) .. (42.64,109.08) -- cycle ;
\draw  [fill={rgb, 255:red, 0; green, 0; blue, 0 }  ,fill opacity=1 ] (69.78,109.05) .. controls (69.78,108.61) and (70.29,108.27) .. (70.92,108.27) .. controls (71.55,108.27) and (72.06,108.61) .. (72.06,109.05) .. controls (72.06,109.48) and (71.55,109.82) .. (70.92,109.82) .. controls (70.29,109.82) and (69.78,109.48) .. (69.78,109.05) -- cycle ;
\draw  [fill={rgb, 255:red, 0; green, 0; blue, 0 }  ,fill opacity=1 ] (78.22,126.65) .. controls (78.22,126.22) and (78.73,125.87) .. (79.36,125.87) .. controls (79.99,125.87) and (80.5,126.22) .. (80.5,126.65) .. controls (80.5,127.08) and (79.99,127.43) .. (79.36,127.43) .. controls (78.73,127.43) and (78.22,127.08) .. (78.22,126.65) -- cycle ;
\draw  [fill={rgb, 255:red, 0; green, 0; blue, 0 }  ,fill opacity=1 ] (56.29,137.57) .. controls (56.29,137.14) and (56.8,136.79) .. (57.44,136.79) .. controls (58.07,136.79) and (58.58,137.14) .. (58.58,137.57) .. controls (58.58,138) and (58.07,138.35) .. (57.44,138.35) .. controls (56.8,138.35) and (56.29,138) .. (56.29,137.57) -- cycle ;
\draw  [fill={rgb, 255:red, 0; green, 0; blue, 0 }  ,fill opacity=1 ] (34.3,126.71) .. controls (34.3,126.28) and (34.81,125.93) .. (35.45,125.93) .. controls (36.08,125.93) and (36.59,126.28) .. (36.59,126.71) .. controls (36.59,127.14) and (36.08,127.49) .. (35.45,127.49) .. controls (34.81,127.49) and (34.3,127.14) .. (34.3,126.71) -- cycle ;
\draw  [fill={rgb, 255:red, 0; green, 0; blue, 0 }  ,fill opacity=1 ] (69.43,90.53) .. controls (69.43,89.85) and (69.98,89.3) .. (70.67,89.3) .. controls (71.35,89.3) and (71.9,89.85) .. (71.9,90.53) .. controls (71.9,91.22) and (71.35,91.77) .. (70.67,91.77) .. controls (69.98,91.77) and (69.43,91.22) .. (69.43,90.53) -- cycle ;
\draw  [fill={rgb, 255:red, 0; green, 0; blue, 0 }  ,fill opacity=1 ] (69.03,50.37) .. controls (69.03,49.69) and (69.58,49.13) .. (70.27,49.13) .. controls (70.95,49.13) and (71.5,49.69) .. (71.5,50.37) .. controls (71.5,51.05) and (70.95,51.61) .. (70.27,51.61) .. controls (69.58,51.61) and (69.03,51.05) .. (69.03,50.37) -- cycle ;
\draw  [fill={rgb, 255:red, 0; green, 0; blue, 0 }  ,fill opacity=1 ] (48.23,70.13) .. controls (48.23,69.45) and (48.78,68.9) .. (49.47,68.9) .. controls (50.15,68.9) and (50.7,69.45) .. (50.7,70.13) .. controls (50.7,70.82) and (50.15,71.37) .. (49.47,71.37) .. controls (48.78,71.37) and (48.23,70.82) .. (48.23,70.13) -- cycle ;
\draw  [fill={rgb, 255:red, 0; green, 0; blue, 0 }  ,fill opacity=1 ] (120.23,69.73) .. controls (120.23,69.05) and (120.78,68.5) .. (121.47,68.5) .. controls (122.15,68.5) and (122.7,69.05) .. (122.7,69.73) .. controls (122.7,70.42) and (122.15,70.97) .. (121.47,70.97) .. controls (120.78,70.97) and (120.23,70.42) .. (120.23,69.73) -- cycle ;
\draw  [fill={rgb, 255:red, 0; green, 0; blue, 0 }  ,fill opacity=1 ] (99.03,50.13) .. controls (99.03,49.45) and (99.58,48.9) .. (100.27,48.9) .. controls (100.95,48.9) and (101.5,49.45) .. (101.5,50.13) .. controls (101.5,50.82) and (100.95,51.37) .. (100.27,51.37) .. controls (99.58,51.37) and (99.03,50.82) .. (99.03,50.13) -- cycle ;
\draw  [fill={rgb, 255:red, 0; green, 0; blue, 0 }  ,fill opacity=1 ] (99.03,90.93) .. controls (99.03,90.25) and (99.58,89.7) .. (100.27,89.7) .. controls (100.95,89.7) and (101.5,90.25) .. (101.5,90.93) .. controls (101.5,91.62) and (100.95,92.17) .. (100.27,92.17) .. controls (99.58,92.17) and (99.03,91.62) .. (99.03,90.93) -- cycle ;
\end{tikzpicture}
\caption*{A graph $F$.}
\label{fig:F7/5}
\endminipage\hfill
\minipage[t]{0.36\textwidth}
\centering
\begin{tikzpicture}[x=0.75pt,y=0.75pt,yscale=-1.5,xscale=1.2]
\draw   (75,49.33) -- (213,49.33) -- (213,78.33) -- (75,78.33) -- cycle ;
\draw   (92.18,121.33) -- (75.85,109.53) -- (82.03,90.36) -- (102.18,90.3) -- (108.45,109.45) -- cycle ;
\draw [color={rgb, 255:red, 254; green, 1; blue, 1 }  ,draw opacity=1 ]   (82.03,90.36) -- (92.18,121.33) ;
\draw    (75.85,109.53) -- (108.45,109.45) ;
\draw    (84.67,71.06) -- (82.03,90.36) ;
\draw   (135.52,121.33) -- (119.19,109.53) -- (125.36,90.36) -- (145.51,90.3) -- (151.79,109.45) -- cycle ;
\draw [color={rgb, 255:red, 254; green, 1; blue, 1 }  ,draw opacity=1 ]   (125.36,90.36) -- (135.52,121.33) ;
\draw    (119.19,109.53) -- (151.79,109.45) ;
\draw    (121.67,74.72) -- (125.36,90.36) ;
\draw   (199.85,121.33) -- (183.52,109.53) -- (189.7,90.36) -- (209.84,90.3) -- (216.12,109.45) -- cycle ;
\draw [color={rgb, 255:red, 254; green, 1; blue, 1 }  ,draw opacity=1 ]   (189.7,90.36) -- (199.85,121.33) ;
\draw    (183.52,109.53) -- (216.12,109.45) ;
\draw    (186,74.39) -- (189.7,90.36) ;
\draw  [fill={rgb, 255:red, 0; green, 0; blue, 0 }  ,fill opacity=1 ] (80.56,90.12) .. controls (80.56,89.44) and (81.11,88.88) .. (81.79,88.88) .. controls (82.48,88.88) and (83.03,89.44) .. (83.03,90.12) .. controls (83.03,90.8) and (82.48,91.36) .. (81.79,91.36) .. controls (81.11,91.36) and (80.56,90.8) .. (80.56,90.12) -- cycle ;
\draw  [fill={rgb, 255:red, 0; green, 0; blue, 0 }  ,fill opacity=1 ] (74.56,109.79) .. controls (74.56,109.1) and (75.11,108.55) .. (75.79,108.55) .. controls (76.48,108.55) and (77.03,109.1) .. (77.03,109.79) .. controls (77.03,110.47) and (76.48,111.02) .. (75.79,111.02) .. controls (75.11,111.02) and (74.56,110.47) .. (74.56,109.79) -- cycle ;
\draw  [fill={rgb, 255:red, 0; green, 0; blue, 0 }  ,fill opacity=1 ] (90.95,121.33) .. controls (90.95,120.65) and (91.5,120.1) .. (92.18,120.1) .. controls (92.87,120.1) and (93.42,120.65) .. (93.42,121.33) .. controls (93.42,122.02) and (92.87,122.57) .. (92.18,122.57) .. controls (91.5,122.57) and (90.95,122.02) .. (90.95,121.33) -- cycle ;
\draw  [fill={rgb, 255:red, 0; green, 0; blue, 0 }  ,fill opacity=1 ] (107.22,109.45) .. controls (107.22,108.77) and (107.77,108.21) .. (108.45,108.21) .. controls (109.14,108.21) and (109.69,108.77) .. (109.69,109.45) .. controls (109.69,110.13) and (109.14,110.68) .. (108.45,110.68) .. controls (107.77,110.68) and (107.22,110.13) .. (107.22,109.45) -- cycle ;
\draw  [fill={rgb, 255:red, 0; green, 0; blue, 0 }  ,fill opacity=1 ] (100.94,90.54) .. controls (100.94,89.86) and (101.49,89.3) .. (102.18,89.3) .. controls (102.86,89.3) and (103.41,89.86) .. (103.41,90.54) .. controls (103.41,91.22) and (102.86,91.77) .. (102.18,91.77) .. controls (101.49,91.77) and (100.94,91.22) .. (100.94,90.54) -- cycle ;
\draw  [fill={rgb, 255:red, 0; green, 0; blue, 0 }  ,fill opacity=1 ] (124.13,90.36) .. controls (124.13,89.67) and (124.68,89.12) .. (125.36,89.12) .. controls (126.04,89.12) and (126.6,89.67) .. (126.6,90.36) .. controls (126.6,91.04) and (126.04,91.59) .. (125.36,91.59) .. controls (124.68,91.59) and (124.13,91.04) .. (124.13,90.36) -- cycle ;
\draw  [fill={rgb, 255:red, 0; green, 0; blue, 0 }  ,fill opacity=1 ] (117.95,109.53) .. controls (117.95,108.85) and (118.5,108.3) .. (119.19,108.3) .. controls (119.87,108.3) and (120.42,108.85) .. (120.42,109.53) .. controls (120.42,110.22) and (119.87,110.77) .. (119.19,110.77) .. controls (118.5,110.77) and (117.95,110.22) .. (117.95,109.53) -- cycle ;
\draw  [fill={rgb, 255:red, 0; green, 0; blue, 0 }  ,fill opacity=1 ] (134.28,121.33) .. controls (134.28,120.65) and (134.83,120.1) .. (135.52,120.1) .. controls (136.2,120.1) and (136.75,120.65) .. (136.75,121.33) .. controls (136.75,122.02) and (136.2,122.57) .. (135.52,122.57) .. controls (134.83,122.57) and (134.28,122.02) .. (134.28,121.33) -- cycle ;
\draw  [fill={rgb, 255:red, 0; green, 0; blue, 0 }  ,fill opacity=1 ] (150.55,109.45) .. controls (150.55,108.77) and (151.1,108.21) .. (151.79,108.21) .. controls (152.47,108.21) and (153.02,108.77) .. (153.02,109.45) .. controls (153.02,110.13) and (152.47,110.68) .. (151.79,110.68) .. controls (151.1,110.68) and (150.55,110.13) .. (150.55,109.45) -- cycle ;
\draw  [fill={rgb, 255:red, 0; green, 0; blue, 0 }  ,fill opacity=1 ] (144.27,90.3) .. controls (144.27,89.62) and (144.83,89.07) .. (145.51,89.07) .. controls (146.19,89.07) and (146.75,89.62) .. (146.75,90.3) .. controls (146.75,90.99) and (146.19,91.54) .. (145.51,91.54) .. controls (144.83,91.54) and (144.27,90.99) .. (144.27,90.3) -- cycle ;
\draw  [fill={rgb, 255:red, 0; green, 0; blue, 0 }  ,fill opacity=1 ] (182.28,109.53) .. controls (182.28,108.85) and (182.84,108.3) .. (183.52,108.3) .. controls (184.2,108.3) and (184.76,108.85) .. (184.76,109.53) .. controls (184.76,110.22) and (184.2,110.77) .. (183.52,110.77) .. controls (182.84,110.77) and (182.28,110.22) .. (182.28,109.53) -- cycle ;
\draw  [fill={rgb, 255:red, 0; green, 0; blue, 0 }  ,fill opacity=1 ] (198.61,121.33) .. controls (198.61,120.65) and (199.17,120.1) .. (199.85,120.1) .. controls (200.53,120.1) and (201.09,120.65) .. (201.09,121.33) .. controls (201.09,122.02) and (200.53,122.57) .. (199.85,122.57) .. controls (199.17,122.57) and (198.61,122.02) .. (198.61,121.33) -- cycle ;
\draw  [fill={rgb, 255:red, 0; green, 0; blue, 0 }  ,fill opacity=1 ] (214.88,109.45) .. controls (214.88,108.77) and (215.44,108.21) .. (216.12,108.21) .. controls (216.8,108.21) and (217.36,108.77) .. (217.36,109.45) .. controls (217.36,110.13) and (216.8,110.68) .. (216.12,110.68) .. controls (215.44,110.68) and (214.88,110.13) .. (214.88,109.45) -- cycle ;
\draw  [fill={rgb, 255:red, 0; green, 0; blue, 0 }  ,fill opacity=1 ] (188.46,90.36) .. controls (188.46,89.67) and (189.01,89.12) .. (189.7,89.12) .. controls (190.38,89.12) and (190.93,89.67) .. (190.93,90.36) .. controls (190.93,91.04) and (190.38,91.59) .. (189.7,91.59) .. controls (189.01,91.59) and (188.46,91.04) .. (188.46,90.36) -- cycle ;
\draw  [fill={rgb, 255:red, 0; green, 0; blue, 0 }  ,fill opacity=1 ] (208.61,90.3) .. controls (208.61,89.62) and (209.16,89.07) .. (209.84,89.07) .. controls (210.53,89.07) and (211.08,89.62) .. (211.08,90.3) .. controls (211.08,90.99) and (210.53,91.54) .. (209.84,91.54) .. controls (209.16,91.54) and (208.61,90.99) .. (208.61,90.3) -- cycle ;
\draw (130,58.4) node [anchor=north west][inner sep=0.75pt]  [font=\large]  {$K_{11}$};
\draw (160.33,106.91) node [anchor=north west][inner sep=0.75pt] [font=\Large]   {$...$};
\end{tikzpicture}
\caption*{A weakly $F$-saturated graph $G$, plus activating edges (shown in red).}
\label{fig:G7/5}
\endminipage\hfill
\minipage[t]{0.35\textwidth}
\centering
\begin{tikzpicture}[x=0.75pt,y=0.75pt,yscale=-1.5,xscale=1.2]
\draw   (95,69.33) -- (233,69.33) -- (233,98.33) -- (95,98.33) -- cycle ;
\draw   (112.18,141.33) -- (95.85,129.53) -- (102.03,110.36) -- (122.18,110.3) -- (128.45,129.45) -- cycle ;
\draw [color={rgb, 255:red, 0; green, 0; blue, 0 }  ,draw opacity=1 ]   (102.03,110.36) -- (112.18,141.33) ;
\draw    (95.85,129.53) -- (128.45,129.45) ;
\draw   (155.52,141.33) -- (139.19,129.53) -- (145.36,110.36) -- (165.51,110.3) -- (171.79,129.45) -- cycle ;
\draw [color={rgb, 255:red, 0; green, 0; blue, 0 }  ,draw opacity=1 ]   (145.36,110.36) -- (155.52,141.33) ;
\draw    (139.19,129.53) -- (171.79,129.45) ;
\draw   (219.85,141.33) -- (203.52,129.53) -- (209.7,110.36) -- (229.84,110.3) -- (236.12,129.45) -- cycle ;
\draw [color={rgb, 255:red, 0; green, 0; blue, 0 }  ,draw opacity=1 ]   (209.7,110.36) -- (219.85,141.33) ;
\draw    (203.52,129.53) -- (236.12,129.45) ;
\draw  [fill={rgb, 255:red, 0; green, 0; blue, 0 }  ,fill opacity=1 ] (100.56,110.12) .. controls (100.56,109.44) and (101.11,108.88) .. (101.79,108.88) .. controls (102.48,108.88) and (103.03,109.44) .. (103.03,110.12) .. controls (103.03,110.8) and (102.48,111.36) .. (101.79,111.36) .. controls (101.11,111.36) and (100.56,110.8) .. (100.56,110.12) -- cycle ;
\draw  [fill={rgb, 255:red, 0; green, 0; blue, 0 }  ,fill opacity=1 ] (94.56,129.79) .. controls (94.56,129.1) and (95.11,128.55) .. (95.79,128.55) .. controls (96.48,128.55) and (97.03,129.1) .. (97.03,129.79) .. controls (97.03,130.47) and (96.48,131.02) .. (95.79,131.02) .. controls (95.11,131.02) and (94.56,130.47) .. (94.56,129.79) -- cycle ;
\draw  [fill={rgb, 255:red, 0; green, 0; blue, 0 }  ,fill opacity=1 ] (110.95,141.33) .. controls (110.95,140.65) and (111.5,140.1) .. (112.18,140.1) .. controls (112.87,140.1) and (113.42,140.65) .. (113.42,141.33) .. controls (113.42,142.02) and (112.87,142.57) .. (112.18,142.57) .. controls (111.5,142.57) and (110.95,142.02) .. (110.95,141.33) -- cycle ;
\draw  [fill={rgb, 255:red, 0; green, 0; blue, 0 }  ,fill opacity=1 ] (127.22,129.45) .. controls (127.22,128.77) and (127.77,128.21) .. (128.45,128.21) .. controls (129.14,128.21) and (129.69,128.77) .. (129.69,129.45) .. controls (129.69,130.13) and (129.14,130.68) .. (128.45,130.68) .. controls (127.77,130.68) and (127.22,130.13) .. (127.22,129.45) -- cycle ;
\draw  [fill={rgb, 255:red, 0; green, 0; blue, 0 }  ,fill opacity=1 ] (120.94,110.54) .. controls (120.94,109.86) and (121.49,109.3) .. (122.18,109.3) .. controls (122.86,109.3) and (123.41,109.86) .. (123.41,110.54) .. controls (123.41,111.22) and (122.86,111.77) .. (122.18,111.77) .. controls (121.49,111.77) and (120.94,111.22) .. (120.94,110.54) -- cycle ;
\draw  [fill={rgb, 255:red, 0; green, 0; blue, 0 }  ,fill opacity=1 ] (144.13,110.36) .. controls (144.13,109.67) and (144.68,109.12) .. (145.36,109.12) .. controls (146.04,109.12) and (146.6,109.67) .. (146.6,110.36) .. controls (146.6,111.04) and (146.04,111.59) .. (145.36,111.59) .. controls (144.68,111.59) and (144.13,111.04) .. (144.13,110.36) -- cycle ;
\draw  [fill={rgb, 255:red, 0; green, 0; blue, 0 }  ,fill opacity=1 ] (137.95,129.53) .. controls (137.95,128.85) and (138.5,128.3) .. (139.19,128.3) .. controls (139.87,128.3) and (140.42,128.85) .. (140.42,129.53) .. controls (140.42,130.22) and (139.87,130.77) .. (139.19,130.77) .. controls (138.5,130.77) and (137.95,130.22) .. (137.95,129.53) -- cycle ;
\draw  [fill={rgb, 255:red, 0; green, 0; blue, 0 }  ,fill opacity=1 ] (154.28,141.33) .. controls (154.28,140.65) and (154.83,140.1) .. (155.52,140.1) .. controls (156.2,140.1) and (156.75,140.65) .. (156.75,141.33) .. controls (156.75,142.02) and (156.2,142.57) .. (155.52,142.57) .. controls (154.83,142.57) and (154.28,142.02) .. (154.28,141.33) -- cycle ;
\draw  [fill={rgb, 255:red, 0; green, 0; blue, 0 }  ,fill opacity=1 ] (170.55,129.45) .. controls (170.55,128.77) and (171.1,128.21) .. (171.79,128.21) .. controls (172.47,128.21) and (173.02,128.77) .. (173.02,129.45) .. controls (173.02,130.13) and (172.47,130.68) .. (171.79,130.68) .. controls (171.1,130.68) and (170.55,130.13) .. (170.55,129.45) -- cycle ;
\draw  [fill={rgb, 255:red, 0; green, 0; blue, 0 }  ,fill opacity=1 ] (164.27,110.3) .. controls (164.27,109.62) and (164.83,109.07) .. (165.51,109.07) .. controls (166.19,109.07) and (166.75,109.62) .. (166.75,110.3) .. controls (166.75,110.99) and (166.19,111.54) .. (165.51,111.54) .. controls (164.83,111.54) and (164.27,110.99) .. (164.27,110.3) -- cycle ;
\draw  [fill={rgb, 255:red, 0; green, 0; blue, 0 }  ,fill opacity=1 ] (202.28,129.53) .. controls (202.28,128.85) and (202.84,128.3) .. (203.52,128.3) .. controls (204.2,128.3) and (204.76,128.85) .. (204.76,129.53) .. controls (204.76,130.22) and (204.2,130.77) .. (203.52,130.77) .. controls (202.84,130.77) and (202.28,130.22) .. (202.28,129.53) -- cycle ;
\draw  [fill={rgb, 255:red, 0; green, 0; blue, 0 }  ,fill opacity=1 ] (218.61,141.33) .. controls (218.61,140.65) and (219.17,140.1) .. (219.85,140.1) .. controls (220.53,140.1) and (221.09,140.65) .. (221.09,141.33) .. controls (221.09,142.02) and (220.53,142.57) .. (219.85,142.57) .. controls (219.17,142.57) and (218.61,142.02) .. (218.61,141.33) -- cycle ;
\draw  [fill={rgb, 255:red, 0; green, 0; blue, 0 }  ,fill opacity=1 ] (234.88,129.45) .. controls (234.88,128.77) and (235.44,128.21) .. (236.12,128.21) .. controls (236.8,128.21) and (237.36,128.77) .. (237.36,129.45) .. controls (237.36,130.13) and (236.8,130.68) .. (236.12,130.68) .. controls (235.44,130.68) and (234.88,130.13) .. (234.88,129.45) -- cycle ;
\draw  [fill={rgb, 255:red, 0; green, 0; blue, 0 }  ,fill opacity=1 ] (208.46,110.36) .. controls (208.46,109.67) and (209.01,109.12) .. (209.7,109.12) .. controls (210.38,109.12) and (210.93,109.67) .. (210.93,110.36) .. controls (210.93,111.04) and (210.38,111.59) .. (209.7,111.59) .. controls (209.01,111.59) and (208.46,111.04) .. (208.46,110.36) -- cycle ;
\draw  [fill={rgb, 255:red, 0; green, 0; blue, 0 }  ,fill opacity=1 ] (228.61,110.3) .. controls (228.61,109.62) and (229.16,109.07) .. (229.84,109.07) .. controls (230.53,109.07) and (231.08,109.62) .. (231.08,110.3) .. controls (231.08,110.99) and (230.53,111.54) .. (229.84,111.54) .. controls (229.16,111.54) and (228.61,110.99) .. (228.61,110.3) -- cycle ;

\draw (150,78.4) node [anchor=north west][inner sep=0.75pt] [font=\large]   {$K_{11}$};
\draw (179.33,126.91) node [anchor=north west][inner sep=0.75pt]  [font=\Large]  {$...$};
\end{tikzpicture}
\caption*{The graph $G'$, a rotation of $G$.}
\label{fig:G'7/5}
\endminipage
\caption{Rotation. The graph $F$ to the left has $w_F = 7/5$~\cite[Section 4]{tz23}. The next graph $G$ is weakly $F$-saturated and attains $e(G) = \frac75|V(G)| + O(1)$; its activating edges are included in $\hat G$ but not in $G$. In the final rotation $G'$, we replace the outgoing edge of each lower component with the activating edge. }
\label{fig:rotation example}
\end{figure}
The next definition applies to any partition of $V(G)$, not necessarily the activation partition.
\begin{definition}[Density]
Let $\P$ be any partition of $V(G)$ and suppose each edge of $G$ is owned by exactly one part in $\P$. Denote $|p|$ to be the number of vertices in the part $p\in \P$ and $m^*(p)$ to be the number of edges in $G$ owned by $p$. Then, the \emph{density} of $p$ is given by $d(p) = m^*(p)/|p|$.
\end{definition}
We emphasize that to talk about the density of a part in a partition $\P$, we require that ownership is defined for \emph{all} edges of $G$. Whenever we use this definition, we provide such an ownership assignment. 

For use in the proof of our main lemma (\cref{lem:goodpartition} below), we now introduce the following notion.

\begin{definition}\label{def:out-connected}
Recall that a \emph{mixed multigraph} $M$ is one with both directed and undirected edges, and its \emph{underlying undirected graph} is the graph on the same vertex set obtained by placing an edge between two vertices if and only if there is at least one edge between them in $M$ (undirected or directed in either direction).

Let $M$ be a mixed multigraph with a distinguished vertex $v$ such that there are no out-edges from $v$ and no undirected edges incident to $v$. Say that $M$ is \emph{out-connected} if after deleting any one out-edge from each vertex of $M$, any two vertices other than $v$ are still connected in the underlying undirected graph.
\end{definition}

\begin{prop}\label{prop:mixed multigraphs}
If a mixed multigraph $M$ with distinguished vertex $v$ is out-connected and has $d_0$ directed edges not to $v$, $d_1$ directed edges to $v$, $e$ undirected edges, and $s$ vertices besides $v$, then at least one of the following inequalities hold: $\frac12 d_0 + e \geq s-1$, or $\frac12 (d_0 + d_1) + e \geq s$.
\end{prop}
\begin{proof}
We proceed by induction on $s$. The base case $s=1$ is trivial, so assume $s>1$ and that the statement is true for smaller $s$. For each vertex $m\in V(M)$ except for $v$, either (i) an undirected edge has $m$ as one end and some $m'\neq m$ as the other, (ii) there are two out-edges from $m$ (possibly to $v$), or else (iii) there are two out-edges from some vertex $m'\neq m$ to $m$. Indeed, if none of these occur, then by deleting an out-edge from each vertex of $M$, we can disconnect $m$ from the rest of the graph: we simply delete the (at most one) out-edge from $m$ and the (at most one) out-edge from each $m'\neq m$ to $m$. 

First suppose case (i) or (iii) holds for some pair $m, m'$. Note that by assumption, neither $m$ nor $m'$ is $v$. Form a mixed multigraph $\tilde M$ by contracting the undirected edge or the two directed edges (in case (i) or (iii) respectively). That is, remove all edges between $m$ and $m'$ and replace those two vertices by a single vertex $\tilde m$; edges with an end in $m$ or $m'$ now have an end in $\tilde m$, with directions inherited from their directions in $M$. No other edges are destroyed by contraction as we allow parallel edges. Our assumptions on $v$ clearly hold for $\tilde M$. Now $\tilde M$ is out-connected since $M$ is: if by removing an out-edge from each vertex of $\tilde M$ we could disconnect two non-$v$ vertices of $\tilde M$, then removing the corresponding edges of $M$ disconnects the corresponding non-$v$ vertices of $M$, a contradiction. Say $\tilde M$ has $d_0'$ directed edges not to $v$, $d_1'$ directed edges to $v$, and $e'$ undirected edges. The induction hypothesis applied to $\tilde M$ tells us that either $\frac12 d_0' + e' \geq (s-1)-1$ or $\frac12(d_0'+d_1') + e' \geq s-1$. Uncontracting the undirected edge(s) or (at least two) directed edges, we have $e \geq e'+1$ or $d_0 \geq d_0' + 2$; in either case, $\frac12 d_0 + e \geq \frac12 d_0' + e' + 1$, and furthermore, $d_1 = d_1'$. Thus either $\frac12d_0 + e \geq s-1$ or $\frac12(d_0+d_1) + e \geq s$, as desired.

Thus we may assume cases (i) and (iii) do not hold for any pair $m, m'$, so in particular $e = 0$. We now show that $d_0 + d_1 \geq 2s$ (the second inequality). If not, then since $d_0 + d_1$ counts the total out-degree of all vertices of $M$, some vertex $m\neq v$ of $M$ has at most one out-edge. But then none of (i), (ii), or (iii) hold for $m$, a contradiction which finishes the proof.
\end{proof}

The next lemma contains the bulk of the proof. It essentially says that we can find a partition of $V(G)$ whose least dense part can be disconnected from its complement via rotation.

\begin{lemma}\label{lem:goodpartition}
Let $k>\delta$ be a fixed integer, let $F$ be a graph with $w_F < \frac\delta2 - \frac1{k+1}$, and let $n_0 \gg_k 1$.\footnote{The notation $n_0\gg_k 1$ means that $n_0$ is a sufficiently large constant, where this constant is allowed to depend on $k$.} Let $G\in \uwSAT(n_0, F)$ and let $\A$ be the activation partition of $G$. Then, there exists a partition $\P$ of $V(G)$ and an ownership assignment of $E(G)$ to parts of $\P$ with the following properties.
\begin{enumerate}[(1)]
\item $\P$ is a coarsening of $\A$, and edges owned by a part $a\in \A$ are owned by the part $p\in \P$ containing $a$.
\item Every edge of $G$ is owned by some $p\in \P$, and each edge of $G$ is incident to the part of $\P$ that owns it.
\item The least dense $p_0\in \P$ has $d(p_0) = \frac\delta2 - \frac1{k'}$, where $k'\in\{|p_0|,2|p_0|\}$, $k'$ is even if $\delta$ is odd, and $k' \leq k$.
\item There is a rotation $G'\in \uwSAT(n_0,F)$ of $G$ for which $p_0$ is a connected component of $G'$.
\end{enumerate}
\end{lemma} 

\begin{proof}

Recall that $\hat G$ is the graph $G$ with activating edges added, and each new edge is owned by its respective part in $\A$ if and only if it has an end in the part (otherwise it is free). Let $\hat G_{\A}$ be the multigraph (allowing parallel edges and self-loops) obtained from $\hat G$ by contracting each part $a\in \A$ to one vertex labeled by $a$; we keep all edges of $\hat G$ and remember which edges are owned by which parts and which are free. An $\A$-matching in $\hat G_\A$ is the image of an $\A$-matching in $\hat G$ under contraction. Define a set $U\subseteq V(\hat G_\A)$ to be \emph{rotation-connected} if for any $\A$-matching $M$ in $\hat G_\A$, every pair of vertices in $U$ is connected in $\hat G_\A - M$ (by a path not necessarily in $U$). Naturally, we call maximal rotation-connected subsets of $V(\hat G_\A)$ \textit{rotation components}. Observe that the property that a pair of vertices $u$ and $v$ are connected in every $\hat G_\A - M$ is an equivalence relation, so each vertex lies in a unique rotation component.

The rotation components thus form a partition $P_\A$ of $\hat G_\A$. Define $\P$ by pulling back $P_\A$ through the contraction map: each part $p\in \P$ consists of the vertices of $G$ that map to a given rotation component in $\P_\A$.

We next define ownership of edges in $\hat G$ for $\P$. 
Each part $p$ of $\P$ owns all edges of $\hat G$ owned by its subparts in the activation partition $\A$, as well as any previously free edges in $\hat G$ with both ends in $p$ (since $E(G) \subseteq E(\hat G)$, ownership of edges of $G$ is inherited). Property 1 of the lemma follows.

Since free edges are never in an $\A$-matching, two parts in  $\A$ with a free edge between them are necessarily in the same part of $\P$, which then owns that edge. Free edges with both ends in a single part $a\in \A$ become owned by the part in $\P$ containing $a$. \cref{prop:ownership} says that all edges are owned by at most one part of $\A$; hence, all edges of $\hat G$ are owned by exactly one part of $\P$. Furthermore, \cref{prop:ownership} says that edges of $G$ owned by some part of $\A$ have at least one end in the part. This establishes property 2 of the lemma. 

We now focus on $p_0$, the part of $\P$ with the lowest density, and prove properties 3 and 4 of the lemma. 
Applying \cref{lem:rotation}, let $G'$ be a weakly $F$-saturated graph which is a rotation of $G$, formed by removing an $\A$-matching $M$ from $\hat G$. When choosing which edges form $M$, whenever possible, choose edges with exactly one end in the distinguished part $p_0$.

To prove property 4, first consider edges in $\hat G$ with one end in $p_0$ which are owned by a different part $p$. 
When forming $M$, we include edges with one end in $p_0$ whenever possible. Thus, all edges owned by $p$ whose other ends are in $p_0$ are included in $M$ and thus removed to form $G'$; indeed, if we could not include all such edges in $M$, then $p$ and $p_0$ would be rotation-connected in $\hat G_\A$, a contradiction. Thus it remains to show that edges in $G'$ owned by $p_0$ have both ends in $p_0$.

We claim that $d(p_0) < \delta/2 - 1/(k+1)$. Recall that every edge in $G$ is owned by exactly one part in $\P$, and $m^*(p)$ counts the number of edges of $G$ owned by the part $p\in \P$; hence, $\sum_{p\in\P} m^*(p)=e(G)$. Suppose $d(p_0) \geq \delta/2 - 1/(k+1)$; since $p_0$ has the lowest density among parts in $\P$, we have
\begin{align*}
e(G) = \sum_{p\in \P} m^*(p)  = \sum_{p\in \P} d(p) |p| \geq \sum_{p\in \P} \left(\frac\delta2 - \frac{1}{k+1}\right) |p| = \left(\frac\delta2-\frac{1}{k+1}\right) n_0.
\end{align*}
But $G\in\uwSAT(n_0,F)$ and $w_F < \delta/2 - 1/(k+1)$, so since $n_0$ is large, this is a contradiction, and the claim is proved. This claim implies in particular that \begin{align}\label{m* upper bound} m^*(p_0) < \left(\frac{\delta}{2} - \frac{1}{k+1}\right)|p_0| < \frac{\delta}{2}|p_0|.
\end{align}

We next prove a lower bound on $m^*(p_0)$. 
Say that $p_0$ consists of $s$ parts $a_1, \ldots, a_s$ from $\mathcal A$. In the $F$-bootstrap percolation process on $G$, immediately after $a_i$ is activated, $a_i$ must own at least $\delta$ edges incident to each vertex of $a_i$ in that new copy of $F$ (where we mean ownership in $\hat G$). In particular, recall that besides the edge that activated $a_i$, edges that are present at this point in the process but not in the base graph $G$ cannot be owned by $a_i$, for otherwise some vertices of $a_i$ would have been activated sooner. Thus, $a_i$ owns at least $\frac{\delta}{2}|a_i| + \frac{1}{2}x(a_i)$ edges in $\hat G$, where $x(a_i)$ is the number of edges $a_i$ owns with only one end in $a_i$. We use the following notation:

\begin{itemize}
    \item $x_0$ is the number of edges in $\hat G$ owned by some $a_i$ with other end in $p_0$ but not in $a_i$;
    \item $x_1$ is the number of edges in $\hat G$ owned by some $a_i$ with other end outside of $p_0$;
    \item $y$ is the number of edges in $\hat G$ with both ends in $p_0$ which are not owned by any $a_i$ (that is, which were free before $p_0$ took ownership).
\end{itemize} 

Recall that each part $a\in \A$ owns at most one more edge in $\hat G$ than in $G$ (its activating edge, if it has at least one end in $a$). Thus, each $a_i$ owns at least $\frac\delta2 |a_i| + \frac12x(a_i) - 1$ edges in $G$. Summing this over $i$ and counting the previously free edges, we obtain
\begin{equation}\label{m* lower bound 1}
m^*(p_0) \geq \frac\delta2 |p_0| + \frac{1}{2} (x_0 + x_1) - s +y.
\end{equation}

\begin{claim*}
We have $\frac12 x_0 + y \geq s-1$.\end{claim*}
\begin{proof} 
Build a mixed multigraph $M$ (see \cref{def:out-connected}) with vertex set $\{a_1, \ldots, a_s\} \cup \{v\}$ corresponding to the subparts of $p_0$ plus a distinguished vertex $v$. Place a directed edge $(a_i, a_j)$ for each edge owned by $a_i$ with other end in $a_j$, and place an undirected edge $\{a_i, a_j\}$ if there is a free edge with an end in each of $a_i$ and $a_j$. Also place a directed edge $(a_i, v)$ for each edge owned by $a_i$ with other end outside of $p_0$.

Observe that since $p_0$ is rotation-connected, $M$ is out-connected. Indeed, recall the choice of $\A$-matching edges owned outside $p_0$, which ensures that all edges into $p_0$ owned outside of $p_0$ are removed. For any choice of one out-edge from each vertex of $M$ which has an out-edge, we can choose the corresponding matching edges owned by the subparts of $p_0$, and choose an internal owned edge for each remaining subpart. By rotation-connectivity, the vertices corresponding to subparts of $p_0$ remain connected after deleting the $\A$-matching. Identifying all subparts outside of $p_0$ with the single vertex $v$, their connecting paths correspond to walks in the underlying undirected graph of $M$ with the chosen out-edges deleted; we deleted all edges with an end in $p_0$ owned outside of $p_0$, and no free edge has one vertex in and one vertex outside of $p_0$. Thus $M$ is out-connected.

Also observe that $M$ has $s$ vertices besides $v$, $x_0$ directed edges not to $v$, $x_1$ directed edges to $v$, and $y$ undirected edges. By \cref{prop:mixed multigraphs}, we have that $\frac12x_0 + y \geq s-1$ or $\frac12(x_0+x_1)+y\geq s$. If the latter inequality holds, then from \eqref{m* lower bound 1}, we obtain $m^*(p_0) \geq \frac\delta2|p_0|$, but this contradicts \eqref{m* upper bound}. Thus we must have $\frac12x_0 + y \geq s-1$, as desired.
\end{proof}

Together with \eqref{m* lower bound 1}, this claim yields \begin{align} \label{m* lower bound 2} m^*(p_0) \geq \frac\delta2 |p_0| + \frac12 x_1 - 1.\end{align}

Combining \eqref{m* upper bound} and \eqref{m* lower bound 2}, we conclude that $x_1$ is $0$ or $1$. That is, $p_0$ only owns at most one edge in $\hat G$ whose other end is outside of $p_0$. If such an edge $e$ is present, say it is owned by the subpart $a_1$. When choosing which edge of $\hat G$ owned by $a_1$ to include in our $\A$-matching $M$, we choose $e$, and $e$ is thus removed when forming $G'$. This concludes the proof of property 4 of the lemma.

Finally, we prove property 3. From \eqref{m* upper bound} and \eqref{m* lower bound 2}, we see that $d(p_0)=m^*(p_0)/|p_0| = \delta/2 - 1/k'$, where $k'$ is either $|p_0|$ or $2|p_0|$, and $1/k' > 1/(k+1)$ so $k'\leq k$. Furthermore, if $\delta$ and $|p_0|$ are both odd, then since $m^*(p_0)$ is an integer, \eqref{m* lower bound 2} tells us that $m^*(p_0) \geq \frac{\delta}{2} |p_0| - \frac{1}{2}$, so we must have $k'=2|p_0|$. That is, $k'$ is even when $\delta$ is odd.
\end{proof}

We can now finish the proof of the only if direction of Theorem \ref{thm:sparse}. The proof also yields \cref{thm:sparse tuza}.

\begin{proof}[Proof of \cref{thm:sparse}, only if direction, and proof of \cref{thm:sparse tuza}]

Suppose $w_F < \delta/2 - 1/(k+1)$, where $k+1$ is even if $\delta$ is odd. We know that $k \geq \delta+1$ by \eqref{eq:delta-bound}. To prove both theorems, we need to show that $\wsat(n, F) \leq (\delta/2 - 1/k)n + O(1)$ if $\delta$ is even, and $\wsat(n, F) \leq (\delta/2 - 1/(k-1))n + O(1)$ if $\delta$ is odd.

Let $n_0\gg_k 1$ be large, and let $G\in\uwSAT(n_0, F)$.
We apply \cref{lem:goodpartition} to obtain the partition $\P$ and graph $G'$ with the properties described in the lemma. 
From property 4 of the lemma, the least dense part $p_0$ of $\P$ is disconnected from its complement in $G'$. Property 1 of the lemma says that $\P$ is a coarsening of $\A$ and inherits ownership from $\A$, and property 2 says that parts of $\P$ only own edges in $G$ with at least one end in the part. Hence, by definition of rotation and $\A$-matchings, it also holds for $\hat G$ and hence for $G'$ that if a part $p\in \P$ owns an edge, then the edge has at least one end in $p$. In particular, $p_0$ owns precisely the edges of $G'$ with both ends in $p_0$. 

By \cref{lem:rotation}, $G'\in\uwSAT(n_0, F)$.
We can now form a sequence of weakly $F$-saturated graphs by copying $p_0$ and its edges, as follows. 

Let $G_0 = G'$ and, for $i\geq 1$, let $G_i$ be the disjoint union of $G_{i-1}$ and a copy of the vertex set $p_0$ and all edges $p_0$ owns. Since $G_{i-1}$ is weakly $F$-saturated, $G_i$ is too: we can first restore the missing edges of $G_{i-1}$, then restore the edges within the new copy of $p_0$ and between the new copy of $p_0$ and the vertices of $G_{i-1}$ besides the old copies of $p_0$ in there, in the same manner they were restored in $G_{i-1}$; finally, we restore edges between the copies of $p_0$, which is easy since at this point both copies are complete to the rest of the graph. This last step relies on a small technicality that since $d(p_0) \ge \delta/2 - 1/|p_0|$, $|p_0| \le k+1$, so $p_0$ is tiny compared to the whole graph $G_i$. 

We have
\begin{align*}
e(G_i) = m^*(p_0) i + |e(G_0)| = \frac{m^*(p_0)}{|p_0|} (|V(G_i)| - V(G_0)|) + |e(G_0)| = d(p_0) |V(G_i)| + O(1).
\end{align*}
Hence, we conclude that $$\wsat(n, F) \leq d(p_0)n + O(1).$$

Property 3 in \cref{lem:goodpartition} tells us that $d(p_0) \leq \delta/2 - 1/k$ if $\delta$ is even, and $d(p_0)\leq \delta/2 - 1/(k-1)$ if $\delta$ is odd. This completes the proof.
\end{proof}

\section{The Dense Regime}\label{dense section}
In this section we prove \cref{thm:dense}.

\begin{proof}[Proof of \cref{thm:dense}]
As discussed in the introduction, $w = \delta - 1$ is obtained by a complete graph on $\delta+1$ vertices. So, we may assume $w < \delta - 1$ and $\delta \geq 3$. 
By \cref{lem:f-tilde}, given $\delta$ and $w = a/b$ a rational number in the desired range, it suffices to construct $F$ with $\delta_F = \delta$ and $\gamma_F = a/b$. 

We construct $F$ with two vertex-disjoint parts, $G$ and $K$. The part $G$ has $k$ vertices, where $k$ is a sufficiently large even multiple of $b$. The part $K$ is a clique of size much greater than $k$ (say, $k^3$ for concreteness).

To describe the edges inside $G$ and between $G$ and $K$, let $\ell$ be the unique integer such that $a/b \in [\delta/2 + \ell/2, \delta/2 + \ell/2+1/2)$; note that $0 \leq \ell \leq \delta-3$. 
Label the vertices of $G$ as $0, 1, \ldots, k-1$. Throughout this proof, labels are understood modulo $k$. For each vertex $i$, place an edge between $i$ and $i+\ell'$ for each $\ell'\in \{1, 2, \ldots, \floor{(\delta-\ell)/2}\}$. If $\delta - \ell$ is odd, also put an edge between $i$ and $i+k/2$ for each $i$. So far, $G$ is a $(\delta - \ell)$-regular Cayley graph. Finally, put $\ell$ edges from each vertex of $G$ to arbitrary vertices of $K$. 

Now we modify the edges as follows. Let $t = (a/b - \delta/2 - \ell/2)k+1$. By definition of $\ell$ and since $k$ is a large even multiple of $b$, we have that $t\in[0, k/2]$ is an integer. If $\delta - \ell$ is odd, pick $t$ indices $i\in\{0,1,\ldots, k/2-1\}$, chosen to be as spread out as possible. To be precise, we choose the indices $i$ to be the integers $\floor{(k/2)/t \cdot j}$ for $j\in\{0, 1, \ldots, t-1\}$. For each of these indices $i$, delete the edge between $i$ and $i+k/2$, and instead put one edge from each of $i$ and $i+k/2$ to arbitrary vertices of $K$. If instead $\delta-\ell$ is even, pick $t$ indices $i\in\{0,1,\ldots, k-1\}$, chosen to be as spread out as possible, i.e., $\floor{k/t \cdot j}$ for $j\in\{0, 1, \ldots, t-1\}$. For each of these indices $i$, delete the edge between $i$ and $i+(\delta-\ell)/2$, and instead put one edge from each of $i$ and $i+(\delta-\ell)/2$ to arbitrary vertices of $K$. See Figure \ref{fig:dense construction} for an example.

\begin{figure}[!htb]
\minipage[t]{0.5\textwidth}
\centering 
\begin{tikzpicture}[x=0.75pt,y=0.75pt,yscale=-1.2,xscale=1.2]

\draw   (400,143) -- (388.54,178.27) -- (358.54,200.06) -- (321.46,200.06) -- (291.46,178.27) -- (280,143) -- (291.46,107.73) -- (321.46,85.94) -- (358.54,85.94) -- (388.54,107.73) -- cycle ;
\draw  [fill={rgb, 255:red, 0; green, 0; blue, 0 }  ,fill opacity=1 ] (288.61,107.06) .. controls (288.61,104.85) and (290.4,103.06) .. (292.61,103.06) .. controls (294.82,103.06) and (296.61,104.85) .. (296.61,107.06) .. controls (296.61,109.27) and (294.82,111.06) .. (292.61,111.06) .. controls (290.4,111.06) and (288.61,109.27) .. (288.61,107.06) -- cycle ;
\draw  [fill={rgb, 255:red, 0; green, 0; blue, 0 }  ,fill opacity=1 ] (383.6,107.1) .. controls (383.6,104.89) and (385.39,103.1) .. (387.6,103.1) .. controls (389.81,103.1) and (391.6,104.89) .. (391.6,107.1) .. controls (391.6,109.31) and (389.81,111.1) .. (387.6,111.1) .. controls (385.39,111.1) and (383.6,109.31) .. (383.6,107.1) -- cycle ;
\draw  [fill={rgb, 255:red, 0; green, 0; blue, 0 }  ,fill opacity=1 ] (317.46,85.94) .. controls (317.46,83.73) and (319.25,81.94) .. (321.46,81.94) .. controls (323.67,81.94) and (325.46,83.73) .. (325.46,85.94) .. controls (325.46,88.15) and (323.67,89.94) .. (321.46,89.94) .. controls (319.25,89.94) and (317.46,88.15) .. (317.46,85.94) -- cycle ;
\draw  [fill={rgb, 255:red, 0; green, 0; blue, 0 }  ,fill opacity=1 ] (354.6,85.9) .. controls (354.6,83.69) and (356.39,81.9) .. (358.6,81.9) .. controls (360.81,81.9) and (362.6,83.69) .. (362.6,85.9) .. controls (362.6,88.11) and (360.81,89.9) .. (358.6,89.9) .. controls (356.39,89.9) and (354.6,88.11) .. (354.6,85.9) -- cycle ;
\draw  [fill={rgb, 255:red, 0; green, 0; blue, 0 }  ,fill opacity=1 ] (276.61,142.06) .. controls (276.61,139.85) and (278.4,138.06) .. (280.61,138.06) .. controls (282.82,138.06) and (284.61,139.85) .. (284.61,142.06) .. controls (284.61,144.27) and (282.82,146.06) .. (280.61,146.06) .. controls (278.4,146.06) and (276.61,144.27) .. (276.61,142.06) -- cycle ;
\draw  [fill={rgb, 255:red, 0; green, 0; blue, 0 }  ,fill opacity=1 ] (288.61,177.06) .. controls (288.61,174.85) and (290.4,173.06) .. (292.61,173.06) .. controls (294.82,173.06) and (296.61,174.85) .. (296.61,177.06) .. controls (296.61,179.27) and (294.82,181.06) .. (292.61,181.06) .. controls (290.4,181.06) and (288.61,179.27) .. (288.61,177.06) -- cycle ;
\draw  [fill={rgb, 255:red, 0; green, 0; blue, 0 }  ,fill opacity=1 ] (383.6,177.1) .. controls (383.6,174.89) and (385.39,173.1) .. (387.6,173.1) .. controls (389.81,173.1) and (391.6,174.89) .. (391.6,177.1) .. controls (391.6,179.31) and (389.81,181.1) .. (387.6,181.1) .. controls (385.39,181.1) and (383.6,179.31) .. (383.6,177.1) -- cycle ;
\draw  [fill={rgb, 255:red, 0; green, 0; blue, 0 }  ,fill opacity=1 ] (317.5,198.9) .. controls (317.5,196.69) and (319.29,194.9) .. (321.5,194.9) .. controls (323.71,194.9) and (325.5,196.69) .. (325.5,198.9) .. controls (325.5,201.11) and (323.71,202.9) .. (321.5,202.9) .. controls (319.29,202.9) and (317.5,201.11) .. (317.5,198.9) -- cycle ;
\draw  [fill={rgb, 255:red, 0; green, 0; blue, 0 }  ,fill opacity=1 ] (354.6,198.6) .. controls (354.6,196.39) and (356.39,194.6) .. (358.6,194.6) .. controls (360.81,194.6) and (362.6,196.39) .. (362.6,198.6) .. controls (362.6,200.81) and (360.81,202.6) .. (358.6,202.6) .. controls (356.39,202.6) and (354.6,200.81) .. (354.6,198.6) -- cycle ;
\draw  [fill={rgb, 255:red, 0; green, 0; blue, 0 }  ,fill opacity=1 ] (394.6,142.1) .. controls (394.6,139.89) and (396.39,138.1) .. (398.6,138.1) .. controls (400.81,138.1) and (402.6,139.89) .. (402.6,142.1) .. controls (402.6,144.31) and (400.81,146.1) .. (398.6,146.1) .. controls (396.39,146.1) and (394.6,144.31) .. (394.6,142.1) -- cycle ;
\draw    (321.46,85.94) -- (358.6,198.6) ;
\draw    (292.61,107.06) -- (387.6,177.1) ;
\draw    (280.61,142.06) -- (398.6,142.1) ;
\draw    (292.61,177.06) -- (387.6,107.1) ;
\draw    (358.54,85.94) -- (321.5,198.9) ;
\draw    (280.61,142.06) -- (321.46,85.94) ;
\draw    (292.61,107.06) -- (358.54,85.94) ;
\draw    (321.46,85.94) -- (387.6,107.1) ;
\draw    (358.6,85.9) -- (398.6,142.1) ;
\draw    (387.6,107.1) -- (387.6,177.1) ;
\draw    (358.6,198.6) -- (398.6,142.1) ;
\draw    (321.5,198.9) -- (387.6,177.1) ;
\draw    (292.61,177.06) -- (358.6,198.6) ;
\draw    (280.61,142.06) -- (321.5,198.9) ;
\draw    (292.61,107.06) -- (292.61,177.06) ;
\draw    (296.89,53.44) -- (292.61,107.06) ;
\draw   (262.89,37.44) -- (412.89,37.44) -- (412.89,62) -- (262.89,62) -- cycle ;
\draw    (306.89,53.44) -- (292.61,107.06) ;
\draw    (315.89,54.44) -- (321.46,85.94) ;
\draw    (323.89,52.44) -- (321.46,85.94) ;
\draw    (355.89,53.44) -- (358.54,85.94) ;
\draw    (363.89,53.44) -- (358.6,85.9) ;
\draw    (371.89,57.44) -- (387.6,107.1) ;
\draw    (381.89,56.44) -- (387.6,107.1) ;
\draw    (389.89,53.44) -- (398.6,142.1) ;
\draw    (395.89,56.44) -- (398.6,142.1) ;
\draw    (387.6,177.1) .. controls (413.89,145.44) and (408.89,97.44) .. (402.89,56.44) ;
\draw    (387.6,177.1) .. controls (413.89,145.44) and (414.89,97.44) .. (408.89,56.44) ;
\draw    (294.6,179.1) .. controls (260.46,147.44) and (273.1,98.44) .. (280.89,57.44) ;
\draw    (294.6,179.1) .. controls (260.46,147.44) and (264.1,98.44) .. (271.89,57.44) ;
\draw    (358.6,198.6) .. controls (432.89,182.44) and (436.89,85.44) .. (406.89,47.44) ;
\draw    (358.6,198.6) .. controls (442.89,191.44) and (440.89,78.44) .. (408.89,41.44) ;
\draw    (321.5,198.9) .. controls (246.89,181.44) and (249.89,90.44) .. (268.89,52.44) ;
\draw    (321.5,198.9) .. controls (232.89,184.44) and (247.89,85.44) .. (266.89,46.44) ;
\draw    (291.5,53.5) -- (280.61,142.06) ;
\draw    (285.89,51.44) -- (280.61,142.06) ;

\draw (332.44,42.4) node [anchor=north west][inner sep=0.75pt]    {$K$};

\end{tikzpicture}
\caption*{Initial construction of $F$ before modification.}
\endminipage\hfill
\minipage[t]{0.5\textwidth}
\centering

\begin{tikzpicture}[x=0.75pt,y=0.75pt,yscale=-1.2,xscale=1.2]

\draw   (400,143) -- (388.54,178.27) -- (358.54,200.06) -- (321.46,200.06) -- (291.46,178.27) -- (280,143) -- (291.46,107.73) -- (321.46,85.94) -- (358.54,85.94) -- (388.54,107.73) -- cycle ;
\draw  [fill={rgb, 255:red, 0; green, 0; blue, 0 }  ,fill opacity=1 ] (288.61,107.06) .. controls (288.61,104.85) and (290.4,103.06) .. (292.61,103.06) .. controls (294.82,103.06) and (296.61,104.85) .. (296.61,107.06) .. controls (296.61,109.27) and (294.82,111.06) .. (292.61,111.06) .. controls (290.4,111.06) and (288.61,109.27) .. (288.61,107.06) -- cycle ;
\draw  [fill={rgb, 255:red, 0; green, 0; blue, 0 }  ,fill opacity=1 ] (383.6,107.1) .. controls (383.6,104.89) and (385.39,103.1) .. (387.6,103.1) .. controls (389.81,103.1) and (391.6,104.89) .. (391.6,107.1) .. controls (391.6,109.31) and (389.81,111.1) .. (387.6,111.1) .. controls (385.39,111.1) and (383.6,109.31) .. (383.6,107.1) -- cycle ;
\draw  [fill={rgb, 255:red, 0; green, 0; blue, 0 }  ,fill opacity=1 ] (317.46,85.94) .. controls (317.46,83.73) and (319.25,81.94) .. (321.46,81.94) .. controls (323.67,81.94) and (325.46,83.73) .. (325.46,85.94) .. controls (325.46,88.15) and (323.67,89.94) .. (321.46,89.94) .. controls (319.25,89.94) and (317.46,88.15) .. (317.46,85.94) -- cycle ;
\draw  [fill={rgb, 255:red, 0; green, 0; blue, 0 }  ,fill opacity=1 ] (354.6,85.9) .. controls (354.6,83.69) and (356.39,81.9) .. (358.6,81.9) .. controls (360.81,81.9) and (362.6,83.69) .. (362.6,85.9) .. controls (362.6,88.11) and (360.81,89.9) .. (358.6,89.9) .. controls (356.39,89.9) and (354.6,88.11) .. (354.6,85.9) -- cycle ;
\draw  [fill={rgb, 255:red, 0; green, 0; blue, 0 }  ,fill opacity=1 ] (276.61,142.06) .. controls (276.61,139.85) and (278.4,138.06) .. (280.61,138.06) .. controls (282.82,138.06) and (284.61,139.85) .. (284.61,142.06) .. controls (284.61,144.27) and (282.82,146.06) .. (280.61,146.06) .. controls (278.4,146.06) and (276.61,144.27) .. (276.61,142.06) -- cycle ;
\draw  [fill={rgb, 255:red, 0; green, 0; blue, 0 }  ,fill opacity=1 ] (288.61,177.06) .. controls (288.61,174.85) and (290.4,173.06) .. (292.61,173.06) .. controls (294.82,173.06) and (296.61,174.85) .. (296.61,177.06) .. controls (296.61,179.27) and (294.82,181.06) .. (292.61,181.06) .. controls (290.4,181.06) and (288.61,179.27) .. (288.61,177.06) -- cycle ;
\draw  [fill={rgb, 255:red, 0; green, 0; blue, 0 }  ,fill opacity=1 ] (383.6,177.1) .. controls (383.6,174.89) and (385.39,173.1) .. (387.6,173.1) .. controls (389.81,173.1) and (391.6,174.89) .. (391.6,177.1) .. controls (391.6,179.31) and (389.81,181.1) .. (387.6,181.1) .. controls (385.39,181.1) and (383.6,179.31) .. (383.6,177.1) -- cycle ;
\draw  [fill={rgb, 255:red, 0; green, 0; blue, 0 }  ,fill opacity=1 ] (317.5,198.9) .. controls (317.5,196.69) and (319.29,194.9) .. (321.5,194.9) .. controls (323.71,194.9) and (325.5,196.69) .. (325.5,198.9) .. controls (325.5,201.11) and (323.71,202.9) .. (321.5,202.9) .. controls (319.29,202.9) and (317.5,201.11) .. (317.5,198.9) -- cycle ;
\draw  [fill={rgb, 255:red, 0; green, 0; blue, 0 }  ,fill opacity=1 ] (354.6,198.6) .. controls (354.6,196.39) and (356.39,194.6) .. (358.6,194.6) .. controls (360.81,194.6) and (362.6,196.39) .. (362.6,198.6) .. controls (362.6,200.81) and (360.81,202.6) .. (358.6,202.6) .. controls (356.39,202.6) and (354.6,200.81) .. (354.6,198.6) -- cycle ;
\draw  [fill={rgb, 255:red, 0; green, 0; blue, 0 }  ,fill opacity=1 ] (394.6,142.1) .. controls (394.6,139.89) and (396.39,138.1) .. (398.6,138.1) .. controls (400.81,138.1) and (402.6,139.89) .. (402.6,142.1) .. controls (402.6,144.31) and (400.81,146.1) .. (398.6,146.1) .. controls (396.39,146.1) and (394.6,144.31) .. (394.6,142.1) -- cycle ;
\draw    (292.61,107.06) -- (387.6,177.1) ;
\draw    (292.61,177.06) -- (387.6,107.1) ;
\draw    (280.61,142.06) -- (321.46,85.94) ;
\draw    (292.61,107.06) -- (358.54,85.94) ;
\draw    (321.46,85.94) -- (387.6,107.1) ;
\draw    (358.6,85.9) -- (398.6,142.1) ;
\draw    (387.6,107.1) -- (387.6,177.1) ;
\draw    (358.6,198.6) -- (398.6,142.1) ;
\draw    (321.5,198.9) -- (387.6,177.1) ;
\draw    (292.61,177.06) -- (358.6,198.6) ;
\draw    (280.61,142.06) -- (321.5,198.9) ;
\draw    (292.61,107.06) -- (292.61,177.06) ;
\draw    (296.89,53.44) -- (292.61,107.06) ;
\draw   (262.89,37.44) -- (412.89,37.44) -- (412.89,62) -- (262.89,62) -- cycle ;
\draw    (306.89,53.44) -- (292.61,107.06) ;
\draw    (315.89,54.44) -- (321.46,85.94) ;
\draw    (323.89,52.44) -- (321.46,85.94) ;
\draw    (355.89,53.44) -- (358.54,85.94) ;
\draw    (363.89,53.44) -- (358.6,85.9) ;
\draw    (371.89,57.44) -- (387.6,107.1) ;
\draw    (381.89,56.44) -- (387.6,107.1) ;
\draw    (389.89,53.44) -- (398.6,142.1) ;
\draw    (395.89,56.44) -- (398.6,142.1) ;
\draw    (387.6,177.1) .. controls (413.89,145.44) and (408.89,97.44) .. (402.89,56.44) ;
\draw    (387.6,177.1) .. controls (413.89,145.44) and (414.89,97.44) .. (408.89,56.44) ;
\draw    (294.6,179.1) .. controls (260.46,147.44) and (273.1,98.44) .. (280.89,57.44) ;
\draw    (294.6,179.1) .. controls (260.46,147.44) and (264.1,98.44) .. (271.89,57.44) ;
\draw    (358.6,198.6) .. controls (432.89,182.44) and (436.89,85.44) .. (406.89,47.44) ;
\draw    (358.6,198.6) .. controls (442.89,191.44) and (440.89,78.44) .. (408.89,41.44) ;
\draw    (321.5,198.9) .. controls (246.89,181.44) and (249.89,90.44) .. (268.89,52.44) ;
\draw    (321.5,198.9) .. controls (232.89,184.44) and (247.89,85.44) .. (266.89,46.44) ;
\draw    (291.5,53.5) -- (280.61,142.06) ;
\draw    (285.89,51.44) -- (280.61,142.06) ;
\draw [color={rgb, 255:red, 255; green, 0; blue, 0 }  ,draw opacity=1 ]   (329.89,54.67) -- (321.46,85.94) ;
\draw [color={rgb, 255:red, 255; green, 0; blue, 0 }  ,draw opacity=1 ]   (350.89,56.67) -- (358.54,85.94) ;
\draw [color={rgb, 255:red, 255; green, 0; blue, 0 }  ,draw opacity=1 ]   (295.89,44.67) -- (280.61,142.06) ;
\draw [color={rgb, 255:red, 255; green, 0; blue, 0 }  ,draw opacity=1 ]   (384.89,47.67) -- (398.6,142.1) ;
\draw [color={rgb, 255:red, 255; green, 0; blue, 0 }  ,draw opacity=1 ]   (321.5,198.9) .. controls (229.89,191.67) and (239.89,71.67) .. (266.89,40.67) ;
\draw [color={rgb, 255:red, 255; green, 0; blue, 0 }  ,draw opacity=1 ]   (358.6,198.6) .. controls (478.89,197.67) and (422.89,19.67) .. (401.89,40.67) ;

\draw (332.44,42.4) node [anchor=north west][inner sep=0.75pt]    {$K$};

\end{tikzpicture}
\caption*{The graph $F$, including the modification.}
\endminipage\hfill
\caption{Construction of $F$ for $\delta = 7$ and $w = \frac{47}{10}$. We have $k=10, \ell = 2,$ and $t = 3$.}
\label{fig:dense construction}
\end{figure}

These modifications do not change the degrees (in $F$) of vertices in $G$, so they each have degree $\delta$. By directly counting edges in our construction, we have $e(G) = (\delta - \ell)k/2 - t$ and $e(G, K) = \ell k + 2t$.\footnote{The notation $e(S,T)$ denotes the number of edges with one end in $S$ and the other in $T$.} Therefore, $$\gamma_F(G) = \frac{(\delta +\ell)k/2 + t - 1}{k} = \frac{(\delta +\ell)k/2 + (a/b-\delta/2-\ell/2)k+1 - 1}{k} = \frac ab.$$
Thus, it remains to show that for any $\emptyset\subsetneq S\subseteq V(F)$, we have $\gamma_F(S) \geq a/b$.

Since $|K|$ is a very large clique, any $S\subseteq V(F)$ containing at least one vertex of $K$ has $\gamma_F(S) \geq |K|/(k+1) > a/b$. Thus we may assume $S\subseteq V(G)$.

Rather than working with $S$ directly, we first show that it suffices to prove $\gamma_F(S') \geq a/b$ for any set $S'$ with certain convenient properties. We justify this in the claim below.

\begin{claim*}
Assume there exists $S$ with $\emptyset \subsetneq S \subseteq V(G)$ with $\gamma_F(S) < a/b$. Then, there exists a set $S'$ with $\emptyset \subsetneq S' \subseteq V(G)$ with $\gamma_F(S') < a/b$ and with the property that $|N_F(v)\sm S'| < \frac{\delta+\ell+1}{2}$ for all $v\in S'$, while $|N_F(v)\sm S'| \geq \frac{\delta+\ell+1}{2}$ for all $v\in V(G)\sm S'$.
\end{claim*}
\begin{proof}
Starting by letting $S' = S$ and modify $S'$ in two stages as follows. In the first stage, while there exists a vertex $v\in V(G)\sm S'$ with $|N_F(v)\sm S'| < \frac{\delta+\ell+1}{2}$, we add such a vertex to $S'$. The first stage ends when there are no more such vertices. In the second stage, while there exists a vertex $v\in S'$ with $|N_F(v)\sm S'| \geq \frac{\delta+\ell+1}2$, we remove such a vertex from $S'$. The second stage ends when there are no more such vertices in $S'$, yielding the final set $S'$.

The final set $S'$ has $|N_F(v)\sm S'| < \frac{\delta+\ell+1}{2}$ for all $v\in S'$; if this were not the case for some $v\in S'$, we would have removed $v$ from $S'$ in the second stage. We also have $|N_F(v)\sm S'| \geq \frac{\delta+\ell+1}{2}$ for all $v\in V(G)\sm S'$. Indeed, consider two cases. If some $v\in V(G)\sm S'$ was removed from $S'$ in the second stage, then at the moment it was removed, it had $|N_F(v)\sm S'| \geq \frac{\delta+\ell+1}{2}$; further removals from $S'$ later in the second stage can only increase $|N_F(v)\sm S'|$. On the other hand, if $v\in V(G)\sm S'$ was not in $S'$ at the beginning of the second stage, the fact that we did not add $v$ to $S'$ in the first stage implies that at the beginning of the second stage we have $|N_F(v)\sm S'| \geq \frac{\delta+\ell+1}{2}$. Once again, the second stage can only increase $|N_F(v)\sm S'|$.

Note that each removal reduces $m_F(S')$ by at least $\delta/2 + \ell/2+1/2$, and each addition increases $m_F(S')$ by at most $\delta/2 + \ell/2$. 
Since $a/b\in[\delta/2 + \ell/2, \delta/2 + \ell/2 + 1/2)$, each of these changes brings $\gamma_F(S') = (m_F(S')-1)/|S'|$ closer to a fraction which is at most $a/b$. So, since $\gamma_F(S) < a/b$, we also have $\gamma_F(S') < a/b$. Note that we never reduce $|S'|$ to $0$ since $|S|\geq 1$ and if ever $|S'|=1$ then $\gamma_F(S') \geq \delta - 1 > a/b$, contradicting that $\gamma_F(S')$ stays below $a/b$.
\end{proof}

Using the claim, to complete the proof of the theorem, it suffices to show that for any $S'$ with $\emptyset \subsetneq S' \subset V(G)$ satisfying $|N_F(v)\sm S'| < \frac{\delta+\ell+1}{2}$ for all $v\in S'$ and $|N_F(v)\sm S'| \geq \frac{\delta+\ell+1}{2}$ for all $v\in V(G)\sm S'$, we have $\gamma_F(S') \geq a/b$.

We next proceed with some casework, which we outline now. In Case 1, we assume that $V(G)\sm S'$ contains many consecutive vertices of $G$, i.e., there is a big gap in $S'$. In this case, we show that $S'$ consists of a collection of ``strings" of vertices, where each string is a large number of vertices in a row and is sufficiently separated from other strings. We then split into two cases based on parity of $\delta - \ell$ (Cases 1a and 1b), in each case using the structure of strings to argue that $\gamma_F(S') \geq a/b$. In Case 2, we assume instead that $S'$ has no large gaps. Here, we show that if $S'$ has many vertices of $G$ in a row (Case 2a), then in fact $S' = G$ (so $\gamma_F(S') = a/b$), while otherwise (Case 2b), there are too many edges between $S'$ and its complement, again yielding $\gamma_F(S') \geq a/b$. 

\textbf{Case 1}. In this case, there is a gap of $\floor{(\delta - \ell)/2}$ consecutive vertices of $G$ which are not in $S'$. Define ``strings" of $S'$ to be the equivalence classes of vertices in $S'$ where $i \sim i'$ if there exists a sequence of vertices $i = i_0, i_1, \ldots, i_p = i'$ with each $i_j \in S'$ and each pair $i_{j}, i_{j+1}$ has difference (modulo $k$) at most $\floor{(\delta - \ell)/2}$. That is, any string in $S'$ is separated by gaps of at least $\floor{(\delta - \ell)/2}$ from its neighboring strings.

\begin{claim*} 
Each string of vertices in $S'$ consists of vertices all in a row. That is, $S'$ has no gaps of size less than $\floor{(\delta - \ell)/2}$. Furthermore, each string contains at least $\floor{(\delta - \ell)/2}+1$ vertices.
\end{claim*} 
\begin{proof}
In any string in $S'$, there must be a left-most vertex $v$ of the string by the assumption of Case 1. We show that the $\floor{(\delta - \ell)/2}$ vertices immediately to the right of $v$ are all in $S'$. If not, then $|N_F(v)\sm S'| \geq \frac{\delta+\ell+1}2$. Indeed, there are $\ell$ edges from $v$ to $K$, and $\floor{(\delta - \ell)/2}$ edges from $v$ to the vertices immediately left of it, all of which are not in $S'$ by definition of $v$. (If the edge from $v$ to $v-\floor{(\delta - \ell)/2}$ was deleted in the modification step of the construction, this count includes instead the extra edge from $v$ to $K$.) Under the assumption that not all of the $\floor{(\delta - \ell)/2}$ vertices immediately to the right of $v$ are in $S'$, there is at least one more vertex in $N_F(v)\sm S'$, adding to a total of at least $\ell + \floor{(\delta - \ell)/2} + 1 \geq \delta/2 + \ell/2 + 1/2$, as desired.

Now, suppose $v$ and the $p-1$ vertices immediately to the right of it are in $S'$, but vertex $v+p$ is not. We have shown that $p-1 \geq \floor{(\delta - \ell)/2}$. We now argue that vertices $v+p+1$ through $v+p+\floor{(\delta - \ell)/2}-1$ are not in $S'$. If any were, then we would have $|N_F(v+p)\sm S'| \leq \delta/2 + \ell/2$, contradicting that $v+p$ is not in $S'$. Indeed, there is necessarily an edge in $F$ from $v+p$ to any vertex in $\{v+p+1, \ldots, v+p+\floor{(\delta - \ell)/2}-1\}$, and a vertex among these in $S'$ detracts from $|N_F(v+p)\sm S'|$. Also, the $\floor{(\delta - \ell)/2}$ edges from $v+p$ to the vertices immediately left of it do not count toward $|N_F(v+p)\sm S'|$, unless $\delta - \ell$ is even and the edge from $v+p$ to $v+p - (\delta - \ell)/2$ was deleted, in which case this count decreases by $1$. 
Thus, if $\delta - \ell$ is odd, then $|N_F(v+p)\sm S'| \leq \delta - (1 + \floor{(\delta - \ell)/2}) \leq \delta/2 + \ell/2$. If $\delta - \ell$ is even, then again $|N_F(v+p)\sm S'| \leq \delta - (1 + (\delta - \ell)/2 - 1) = \delta/2 + \ell/2$, as desired.

Thus, the vertices from $v+p$ to $v+\floor{(\delta - \ell)/2}-1$ are all not in $S'$, showing that the vertices from $v$ to $v+p-1$ are the entirety of this string.
\end{proof}

The proof for Case 1 now splits into two subcases depending on the parity of $\delta - \ell$.

\textbf{Case 1a:} $\delta - \ell$ is even. Recall that then there are no edges between opposite vertices of $G$. Thus, different strings of $S'$ constitute different connected components of $G[S']$. Since we are trying to show that $\gamma_F(S') = (m_F(S') - 1)/|S'|$ is at least $a/b$, it suffices to show that $\gamma_F(C) \leq a/b$ for every connected component $C$ of $G[S']$. Indeed, restricting to the connected component $C$ of $G[S']$ minimizing $m_F(C) / |C|$, we have $\gamma_F(C) \leq \gamma_F(S')$ since $|C| \leq |S'|$. For the remainder of Case 1a, fix such a component, and for simplicity call its vertex set again $S'$. By the claim above, $S'$ is a string of vertices with size $p:=|S'| \geq \floor{(\delta-\ell)/2}+1$. 

Then, $m_F(S')$ counts $\ell p$ edges to $K$, and $p+\ell'$ edges between pairs of vertices whose labels differ by $\ell'$, for each $\ell'\in\{1, 2, \ldots, (\delta - \ell)/2\}$. Furthermore, there are $L:=p - (\delta - \ell)/2$ edges between pairs of vertices both in the string which could be deleted in the modification step of the construction of $F$ and replaced with an edge from each of the ends to $K$, increasing $m_F(S')$ by $1$.

We lower bound how many of these $L$ edges are deleted. These edges have $L$ consecutive labels, say $q,\ldots,q+L-1$ (modulo $k$), and the deleted edges have labels $\floor{k/t\cdot j}$ for $j\in\{0, 1, \ldots, t-1\}$. Since $q \leq \floor{k/t\cdot j} \leq q+L-1$ if and only if $qt/k \leq j < (q+L)t/k$ 
(interpreting this interval modulo $t$), the number of deleted edges out of these $L$ is exactly
\[ \ceil{\frac{(q+L)t}{k}} - \ceil{\frac{qt}{k}} \geq \floor{\frac{Lt}{k}}. \]
Substituting $L = p-(\delta-\ell)/2$, the above is at least $\floor{(p-(\delta-\ell)/2)\cdot t /k} \geq (p-(\delta-\ell)/2)\cdot t/k - 1 \geq p t/k - (\delta-\ell)/4 - 1$, where we used that $t\leq k/2$.

Thus, \begin{align*} 
m_F(S') &\geq \ell p + \sum_{\ell' = 1}^{(\delta-\ell)/2}(p+\ell') + p t/k - (\delta-\ell)/4 - 1
\\& = p(\ell + (\delta - \ell)/2 + t/k) + \frac{1}{2}\cdot \frac{\delta - \ell}{2}\left(\frac{\delta - \ell}2+1\right)- (\delta-\ell)/4 - 1
\\&= p(\delta/2 + \ell/2 + t/k) + (\delta-\ell)^2/8 - 1
\\&\geq p(\delta/2 + \ell/2 + t/k) + 1. 
\end{align*}
For the last inequality, recall that $\ell \leq \delta - 3$, so the assumption that $\delta - \ell$ is even implies that $\delta - \ell \geq 4$. 

This implies that $\gamma_F(S') = (m_F(S')-1)/|S'| \geq \delta / 2 + \ell / 2 + t/k > \gamma_F(G) = a/b$, finishing the proof in this subcase.

\textbf{Case 1b:} $\delta-\ell$ is odd. 
We first consider an easy subcase: suppose $S'$ has only one string on $p \leq k/2$ vertices. 
Note that in this subcase, $S'$ does not contain any edges between pairs $i$ and $i+k/2$. So, when the modification step of the construction replaces some of these edges with edges from $i$ and from $i+k/2$ to $K$, $m_F(S')$ remains unchanged. There are $p$ of these kinds of edges incident to $S'$. There are also $\ell p$ other edges from $S'$ to $K$, as well as $p+\ell'$ edges between pairs $i$ and $i+\ell'$, for each $\ell'\in \{1, 2, \ldots, (\delta - \ell - 1)/2\}$. In total,
\begin{align*}m_F(S') &= p + \ell p + \sum_{\ell'=1}^{(\delta - \ell - 1)/2} (p+\ell')
\\&= p(1+\ell+(\delta - \ell - 1)/2) + \frac{1}{2}\frac{\delta-\ell-1}{2}\left(\frac{\delta-\ell-1}2+1\right)
\\&\geq p(\delta/2 + \ell/2 + 1/2) + 1,
\end{align*}
where in the last inequality we used that $\ell \leq \delta - 3$.
Thus $\gamma_F(S') = (m_F(S') - 1)/|S'| \geq \delta/2 + \ell/2 + 1/2 \geq a/b$.

For the rest of Case 1b, rather than counting $m_F(S')$, we count a very similar quantity $\tilde m_F(S')$, where for any edge in $F$ between a pair $i$ and $i+k/2$, $\tilde m_F(S')$ counts $1/2$ for each end of the edge which appears in $S'$. Other kinds of edges are counted toward $\tilde m_F(S')$ just as they are counted toward $m_F(S')$. Clearly $\tilde m_F(S') \leq m_F(S')$, so showing that $(\tilde m_F(S') - 1)/|S'| \geq a/b$ is still a contradiction. However, $\tilde m_F(S')$ is easier to bound in what follows as it allows us to track vertices more individually.

Consider a string $T$ of $S'$ with $p$ vertices, say with labels $q,\ldots,q+p-1$ (modulo $k$). The ends of the deleted edges have labels $\floor{(k/2)/t\cdot j}$ for $j\in\{0, 1, \ldots, 2t-1\}$. Since $q\leq \floor{(k/2)/t\cdot j} \leq q+p-1$ if and only if $2qt/k\leq j < 2(q+p)t/k$ (interpreting this interval modulo $2t$), the number of vertices of $T$ whose edges to their opposite vertices were deleted is exactly
\[\ceil{\frac{2(q+p)t}{k}} - \ceil{\frac{2qt}{k}} \geq \floor{\frac{2pt}{k}} \geq \frac{2pt}{k}-1.\]
For these vertices, the replacement edge to $K$ counts $1$ toward $\tilde m_F(T)$, while for the other vertices, the edge to the opposite vertex counts $1/2$ toward $\tilde m_F(T)$. We thus have
\begin{align}\label{eq: tilde m} \tilde m_F(T) &= \ell p + \sum_{\ell'=1}^{(\delta - \ell - 1)/2} (p + \ell') + \frac{1}{2}p + \frac{1}{2} (2pt/k - 1)\nonumber
\\&\geq p(\ell + (\delta - \ell - 1)/2 + 1/2 + t/k) + \frac{1}{2}\cdot\frac{\delta-\ell-1}{2}\left(\frac{\delta-\ell-1}2+1\right) - 1/2\nonumber
\\&\geq p(\delta/2 + \ell/2 +t/k) + 1/2.
\end{align}

We now take on a second subcase: if $S'$ has a single string with $p > k/2$ vertices, 
then \eqref{eq: tilde m} gives $\gamma_F(S') \geq (\tilde m_F(S')-1)/|S'|= \delta / 2 + \ell/2 + t/k - 1/(2p) \geq \delta / 2 + \ell/2 + t/k - 1/k = a/b$.

Finally, we consider the only remaining case, where $S'$ has at least two strings, say $T_1, \ldots, T_r$ where $r\geq 2$. By definition of strings and $\tilde m_F$, we have $\tilde m_F(S') = \tilde m_F(T_1)+\ldots + \tilde m_F(T_r)$. Then, \eqref{eq: tilde m} tells us that $\tilde m_F(S') \geq |S'|(\delta/2 + \ell/2 + t/k) + 1$, and so $\gamma_F(S) \geq \delta/2 + \ell/2 + t/k \geq a/b$.

This concludes Case 1.

\textbf{Case 2}. In this case, any gap in $S'$ has at most $\floor{(\delta-\ell)/2} - 1$ vertices. Again, we split into two subcases.

\textbf{Case 2a:} There exists a set of at least $\floor{(\delta-\ell-1)/2}$ consecutive vertices in $S'$. 
Say the vertices labeled $q, q+1, \ldots, q+p-1$ (modulo $k$) are in $S'$ but $q+p$ is not, where $p\geq \floor{(\delta-\ell-1)/2}$. 
We show that $|N_F(q+p) \sm S'| \leq (\delta+\ell)/2$, contradicting that $q+p\not\in S'$ and showing that in fact all the vertices $0, \ldots, k-1$ of $G$ are in $S'$. This is a contradiction to $\gamma_F(S') < a/b$ since $\gamma_F(G) = a/b$.

To prove that $|N_F(q+p) \sm S'| \leq (\delta+\ell)2$, first note that since $p \geq \floor{(\delta-\ell-1)/2}$, all of the $\floor{(\delta-\ell-1)/2}$ vertices immediately to the left of $q+p$ are already in $S'$, and in our construction they all have edges with $q+p$. 
These vertices do not count toward $|N_F(q+p)\sm S'|$.
Further, since there are no gaps of size $\floor{(\delta-\ell)/2}$ in $S'$, some vertex with label between $q+p+1$ and $q+p+\floor{(\delta-\ell)/2} - 1$ must be in $S'$. Since this vertex  also does not count toward $|N_F(q+p)\sm S'|$, we see that $|N_F(q+p)\sm S'| \leq \delta - (\floor{(\delta-\ell-1)/2}) - 1 \leq \delta/2 + \ell/2$.

\textbf{Case 2b:} Any set of consecutive vertices in $S'$ has size at most $\floor{(\delta-\ell-1)/2}-1$. 
Note that $m_F(S')$ equals the sum over vertices of $S'$ of $1$ times the number of edges from that vertex to a vertex not in $S'$ plus $1/2$ times the number of edges from that vertex to a vertex in $S'$. 

Consider a vertex $v\in S'$. Recall that $v$ has an edge to every vertex with label at most $\floor{(\delta-\ell-1)/2}$ away. Since there are no sets of $\floor{(\delta - \ell - 1)/2}$ consecutive vertices in $S'$, we see that $v$ has an edge to some vertex to its left not in $S'$ with label at most $\floor{(\delta-\ell-1)/2}$ away, as well as an edge to some vertex to its right not in $S'$ with label at most $\floor{(\delta-\ell-1)/2}$ away. The vertex $v$ also has $\ell$ edges to vertices in $K$, which are not in $S'$. 
Thus, counting all the edges with exactly one end in $S'$ and with both ends in $S'$, we obtain $m_F(S') \geq |S'| \left((2 + \ell) + 1/2 \cdot (\delta - 2 - \ell)\right) = |S'|(\delta/2 + \ell/2 + 1)$. Clearly $|S'| \geq 2$, so $\gamma_F(S') \geq \delta/2 + \ell/2 + 1 - 1/2 \geq a/b$. This concludes the proof.
\end{proof}

\begin{remark}
The construction in the proof of \cref{thm:dense} is completely deterministic. An earlier version of this paper proved a weaker version of the theorem by taking $G$ to be a random regular graph and using its edge expansion.
\end{remark}

\section{Concluding remarks}\label{conclusion}
\subsection{Tuza's conjecture and a question of Terekhov and Zhukovskii}
We first return to Tuza's \cref{conj:tuza}, which is still open when $w_F \in [\delta_F/2, \delta_F-1]$. We remark that for all of the graphs $F$ we construct in the proofs in this paper, we indeed have $\wsat(n, F) = w_F n + O(1)$. This is easy to see since the proof of our main tool for constructions, \cref{lem:f-tilde}, already shows that the graph $\tilde F$ guaranteed by that lemma has $\wsat(n, \tilde F) = \gamma_F n + O(1)$.

Terekhov and Zhukovskii \cite{tz23} gave some intuition for Tuza's conjecture by remarking that for many graphs $F$, an asymptotically minimum sequence of weakly $F$-saturated graphs can be formed by repeatedly copying some part of $F$ minus an edge. They asked if $w_F$ is always equal to $\gamma_F(S)$ for some $S\subseteq V(F)$.

Here we answer their question in the negative. Let $H$ be the Cayley graph $\Gamma(\Z/7, \{\pm1, \pm2\})$, with one additional edge (say from vertex 0 to vertex 3). Let $K_k$ denote a clique of size $k$. Form the graph $F$ by taking the disjoint union of $H$, $K_7$, and $K_{100}$, and then placing one edge from each of two distinct vertices of the $K_7$ to a vertex of the $K_{100}$ (see Figure \ref{fig:F15/7}). Observe that no $S\subseteq V(F)$ attains $\gamma_F(S) = 15/7$. 

\begin{figure}[!htb]
\centering
\begin{tikzpicture}[x=0.75pt,y=0.75pt,yscale=-2,xscale=2]
\draw   (149.37,133.18) -- (135.83,150.13) -- (114.14,150.12) -- (100.62,133.14) -- (105.47,112) -- (125.02,102.6) -- (144.56,112.03) -- cycle ;
\draw   (218.16,133.13) -- (204.62,150.08) -- (182.92,150.07) -- (169.41,133.1) -- (174.25,111.95) -- (193.81,102.55) -- (213.35,111.98) -- cycle ;
\draw   (230,100.6) -- (279,100.6) -- (279,150.6) -- (230,150.6) -- cycle ;
\draw  [fill={rgb, 255:red, 0; green, 0; blue, 0 }  ,fill opacity=1 ] (122.72,102.6) .. controls (122.72,101.33) and (123.75,100.3) .. (125.02,100.3) .. controls (126.29,100.3) and (127.32,101.33) .. (127.32,102.6) .. controls (127.32,103.87) and (126.29,104.9) .. (125.02,104.9) .. controls (123.75,104.9) and (122.72,103.87) .. (122.72,102.6) -- cycle ;
\draw  [fill={rgb, 255:red, 0; green, 0; blue, 0 }  ,fill opacity=1 ] (103.17,112) .. controls (103.17,110.73) and (104.2,109.7) .. (105.47,109.7) .. controls (106.74,109.7) and (107.77,110.73) .. (107.77,112) .. controls (107.77,113.27) and (106.74,114.3) .. (105.47,114.3) .. controls (104.2,114.3) and (103.17,113.27) .. (103.17,112) -- cycle ;
\draw  [fill={rgb, 255:red, 0; green, 0; blue, 0 }  ,fill opacity=1 ] (98.32,133.14) .. controls (98.32,131.87) and (99.35,130.84) .. (100.62,130.84) .. controls (101.89,130.84) and (102.92,131.87) .. (102.92,133.14) .. controls (102.92,134.42) and (101.89,135.44) .. (100.62,135.44) .. controls (99.35,135.44) and (98.32,134.42) .. (98.32,133.14) -- cycle ;
\draw  [fill={rgb, 255:red, 0; green, 0; blue, 0 }  ,fill opacity=1 ] (111.84,150.12) .. controls (111.84,148.85) and (112.87,147.82) .. (114.14,147.82) .. controls (115.41,147.82) and (116.44,148.85) .. (116.44,150.12) .. controls (116.44,151.39) and (115.41,152.42) .. (114.14,152.42) .. controls (112.87,152.42) and (111.84,151.39) .. (111.84,150.12) -- cycle ;
\draw  [fill={rgb, 255:red, 0; green, 0; blue, 0 }  ,fill opacity=1 ] (142.26,112.03) .. controls (142.26,110.76) and (143.29,109.73) .. (144.56,109.73) .. controls (145.83,109.73) and (146.86,110.76) .. (146.86,112.03) .. controls (146.86,113.3) and (145.83,114.33) .. (144.56,114.33) .. controls (143.29,114.33) and (142.26,113.3) .. (142.26,112.03) -- cycle ;
\draw  [fill={rgb, 255:red, 0; green, 0; blue, 0 }  ,fill opacity=1 ] (147.07,133.18) .. controls (147.07,131.91) and (148.1,130.88) .. (149.37,130.88) .. controls (150.64,130.88) and (151.67,131.91) .. (151.67,133.18) .. controls (151.67,134.45) and (150.64,135.48) .. (149.37,135.48) .. controls (148.1,135.48) and (147.07,134.45) .. (147.07,133.18) -- cycle ;
\draw  [fill={rgb, 255:red, 0; green, 0; blue, 0 }  ,fill opacity=1 ] (133.53,150.13) .. controls (133.53,148.86) and (134.56,147.83) .. (135.83,147.83) .. controls (137.1,147.83) and (138.13,148.86) .. (138.13,150.13) .. controls (138.13,151.4) and (137.1,152.43) .. (135.83,152.43) .. controls (134.56,152.43) and (133.53,151.4) .. (133.53,150.13) -- cycle ;
\draw  [fill={rgb, 255:red, 0; green, 0; blue, 0 }  ,fill opacity=1 ] (191.51,102.55) .. controls (191.51,101.28) and (192.54,100.25) .. (193.81,100.25) .. controls (195.08,100.25) and (196.11,101.28) .. (196.11,102.55) .. controls (196.11,103.82) and (195.08,104.85) .. (193.81,104.85) .. controls (192.54,104.85) and (191.51,103.82) .. (191.51,102.55) -- cycle ;
\draw  [fill={rgb, 255:red, 0; green, 0; blue, 0 }  ,fill opacity=1 ] (211.05,111.98) .. controls (211.05,110.71) and (212.07,109.68) .. (213.35,109.68) .. controls (214.62,109.68) and (215.65,110.71) .. (215.65,111.98) .. controls (215.65,113.25) and (214.62,114.28) .. (213.35,114.28) .. controls (212.07,114.28) and (211.05,113.25) .. (211.05,111.98) -- cycle ;
\draw  [fill={rgb, 255:red, 0; green, 0; blue, 0 }  ,fill opacity=1 ] (215.86,133.13) .. controls (215.86,131.86) and (216.89,130.83) .. (218.16,130.83) .. controls (219.43,130.83) and (220.46,131.86) .. (220.46,133.13) .. controls (220.46,134.4) and (219.43,135.43) .. (218.16,135.43) .. controls (216.89,135.43) and (215.86,134.4) .. (215.86,133.13) -- cycle ;
\draw  [fill={rgb, 255:red, 0; green, 0; blue, 0 }  ,fill opacity=1 ] (202.32,150.08) .. controls (202.32,148.81) and (203.35,147.78) .. (204.62,147.78) .. controls (205.89,147.78) and (206.92,148.81) .. (206.92,150.08) .. controls (206.92,151.35) and (205.89,152.38) .. (204.62,152.38) .. controls (203.35,152.38) and (202.32,151.35) .. (202.32,150.08) -- cycle ;
\draw  [fill={rgb, 255:red, 0; green, 0; blue, 0 }  ,fill opacity=1 ] (180.62,150.07) .. controls (180.62,148.8) and (181.65,147.77) .. (182.92,147.77) .. controls (184.19,147.77) and (185.22,148.8) .. (185.22,150.07) .. controls (185.22,151.34) and (184.19,152.37) .. (182.92,152.37) .. controls (181.65,152.37) and (180.62,151.34) .. (180.62,150.07) -- cycle ;
\draw  [fill={rgb, 255:red, 0; green, 0; blue, 0 }  ,fill opacity=1 ] (171.95,111.95) .. controls (171.95,110.68) and (172.98,109.65) .. (174.25,109.65) .. controls (175.52,109.65) and (176.55,110.68) .. (176.55,111.95) .. controls (176.55,113.22) and (175.52,114.25) .. (174.25,114.25) .. controls (172.98,114.25) and (171.95,113.22) .. (171.95,111.95) -- cycle ;
\draw  [fill={rgb, 255:red, 0; green, 0; blue, 0 }  ,fill opacity=1 ] (167.11,133.1) .. controls (167.11,131.82) and (168.14,130.8) .. (169.41,130.8) .. controls (170.68,130.8) and (171.71,131.82) .. (171.71,133.1) .. controls (171.71,134.37) and (170.68,135.4) .. (169.41,135.4) .. controls (168.14,135.4) and (167.11,134.37) .. (167.11,133.1) -- cycle ;
\draw    (218.16,133.13) -- (238.53,138.13) ;
\draw    (213.35,111.98) -- (240.13,113.73) ;
\draw    (105.47,112) -- (144.56,112.03) ;
\draw    (125.02,102.6) -- (149.37,133.18) ;
\draw    (144.56,112.03) -- (135.83,150.13) ;
\draw    (100.62,133.14) -- (135.83,150.13) ;
\draw    (100.62,133.14) -- (125.02,102.6) ;
\draw    (114.14,150.12) -- (149.37,133.18) ;
\draw    (105.47,112) -- (114.14,150.12) ;
\draw    (125.02,102.6) -- (135.83,150.13) ;
\draw (187.6,124.13) node [anchor=north west][inner sep=0.75pt]   [font=\Large] {$K_{7}$};
\draw (244.4,120.93) node [anchor=north west][inner sep=0.75pt]  [font=\Large]  {$K_{100}$};
\draw (116.8,122.73) node [anchor=north west][inner sep=0.75pt]  {$H$};
\end{tikzpicture}
\caption{The graph $F$.}
\label{fig:F15/7}
\end{figure}

\begin{prop}\label{prop:Fgraph}
    For the graph $F$ described above, we have $w_F = 15/7$.
\end{prop}
\begin{proof}
To see that $w_F\leq 15/7$, form a sequence of graphs by taking $G_0$ to be the clique $K_{107}$ and, for $i\geq1$, letting $G_i$ be the disjoint union of $G_{i-1}$ and a Cayley graph $\Gamma_i = \Gamma(\Z/7, \{\pm1, \pm2\})$, with one edge between a vertex of $\Gamma_i$ and a vertex of $G_0$. Observe that each step introduces $15$ edges and $7$ vertices, so $\lim_{i\to\infty} e(G_i)/|V(G_i)| = 15/7$. 

We claim that $G_i$ is weakly $F$-saturated for $i\geq 1$. We proceed by induction on $i$; in the base case of $i=1$, we can first restore the edges inside $\Gamma_1$, forming a copy of $H$ in $G_1$ (and hence a new copy of $F$) with each edge restored. Then we place all edges from $\Gamma_1$ to $G_0$; each such edge forms a new copy of the $K_7$ attached to the $K_{100}$ by two edges which appear in $F$. For the induction step, starting with $G_i$, first restore the edges inside $G_{i-1}$ using the induction hypothesis. Then, continue in a similar manner as the base case: restore the edges inside $\Gamma_i$, the new copy of the Cayley graph, then place all edges from the $\Gamma_i$ to $G_{i-1}$.

We prove $w_F\ge 15/7$ in several steps. Let $G\in\uwSAT(n, F)$. Observe that each component of our graph $F$ is $2$-edge-connected. So, if $G$ were disconnected it would be impossible to restore edges between its components. Thus $G$ is connected.

Consider the activation partition $\mathcal A$ and the edge ownership assignment of $G$ (see \cref{sparse section}). We then reassign the ownership of some edges as follows. During the $F$-bootstrap percolation process on $G$, we build a spanning tree $T$ of $G$, described in the proof of \cref{claim:ownedoutsideT} below. Let $r\in V(G)$ be an arbitrary root vertex. For any edge $e\in E(T)$, let the \emph{bottom} of $e$ be the end of $e$ further from $r$. For each edge $e$ in $T$, cancel the ownership of $e$ (if any), then declare that $e$ is now owned by the part $a\in \mathcal A$ which contains the bottom of $e$.

The following two claims bound the number of edges owned by each part of $\mathcal A$.

\begin{claim}\label{claim:ownedinT}
Under the new definition of ownership, each $a\in \mathcal A$ owns exactly $|a|-\mathbbm1_{r\in a}$ edges of $T$.
\end{claim}
\begin{proof}
For each $a\in \mathcal A$, the number of edges of $T$ owned by $a$ is given by 
\begin{align*} \sum_{v\in a} |\{e\in e(T): v\text{ is the bottom of $e$}\}| = \sum_{v\in a} \mathbbm 1_{v\neq r},\end{align*} where to obtain the equality we note that if $v=r$ then the summand is $0$, while otherwise the summand only counts the last edge in the unique path from $r$ to $v$ in $T$. The claim follows.
\end{proof}
Recall that when a part $a\in \mathcal A$ is activated, those vertices are part of a new copy of $F$ in the current graph. In a mild abuse of notation, we let $a$ denote a subset of vertices in $F$, corresponding to the subset activated in the new copy of $F$. 
Let $c(F-a)$ then denote the number of components of the graph $F-a$.
\begin{claim}\label{claim:ownedoutsideT}
We may choose the spanning tree $T$ so that each $a\in \mathcal A$ owns at least $m_F(a) + 1 - |a| - c(F-a)$ edges not in $T$.
\end{claim}
\begin{proof}
We build up $T$ as follows. Let $U\subseteq V(G)$ and the eventual spanning tree $T$ both start empty. In each step of the $F$-bootstrap percolation process on $G$, add to $U$ the newly activated vertices (if any), and then add edges (chosen arbitrarily) of $G[U]$ to $T$ that do not create a cycle in $T$, until no such edges are left.
Since $G$ is connected and all vertices are eventually active, $T$ is a spanning tree at the end of the process. 

To prove the desired bound on the number of edges not in $T$ owned by $a$, consider the step when $a$ is activated. Let $F'\subseteq G$ be the new copy of $F$ formed, minus the newly restored edge which activated $a$. Every edge of $F'$ incident to $a$ lies in $G$; indeed, an edge restored earlier in the process has both ends active at the moment it is restored.  

Observe that right before $a$ is activated, any two vertices in a single component of $F'-a$ are connected by a path in $T$. Indeed, the vertices of $F'-a$ were previously activated, so we added as many edges as possible with both ends in $F'-a$ to $T$. For two vertices $u, v$ in a single component of $F'-a$, let $u = w_0, w_1, \ldots, w_t= v$ be a path between them with all $w_i\in F'-a$. If there is no path in $T$ from $u$ to $v$, then for some $i$ there is no path in $T$ from $w_i$ to $w_{i+1}$, but then we could have added $w_iw_{i+1}$ to $T$ without creating a cycle in $T$.

We form an auxiliary graph $L$ which will act as a contraction of $F'$. Label the components of $F'-a$ with indices $1, \ldots, c(F'-a)$, and define new vertices $b_1, \ldots, b_{c(F'-a)}$, each representing a component. Let $V(L) = a \cup \bigcup_{i=1}^{c(F'-a)} b_i$. The edges of $L$ correspond to the edges in $T$ owned by $a$, as follows. For $u, v\in a$, we put $uv \in E(L)$ if $uv\in E(G)$ is owned by $a$ and $uv$ was placed into $T$; and for $u\in a$ and a vertex $b_i$, we put  $ub_i \in E(L)$ if there is an edge $uv \in E(G)$ which is owned by $a$ and in $T$ for \emph{some} vertex $v$ in component $i$ of $F'-a$. There are no edges $b_i b_j$ in $L$ for any $i,j$. By the observation in the previous paragraph, there cannot be two edges in $T$ between a vertex $u\in a$ and distinct vertices $v, v'$ in component $i$ of $F'-a$ because $T$ already has a path between $v$ and $v'$ before $a$ is activated. Furthermore, recall that edges owned by $a$ are part of $F'$ and have at least one end in $a$. Thus, $e(L)$ equals the number of edges in $T$ owned by $a$.

We note that $L$ is a forest. Indeed, if $L$ had a cycle $C$, then we could build a cycle $C'$ in $T$, as follows. Traverse $C$ in $L$ starting from a vertex $u_0 \in V(C) \cap a$. For any edge of $C$ with both ends in $a$, that same edge is in $T$ and we add it to $C'$. If $C$ has an edge $e$ from a vertex $u\in a$ to a vertex $b_i$, then we add to $C'$ an edge of $T$ between $u$ and a vertex $v$ in component $i$ of $F'-a$, which exists by construction of $L$. The edge immediately after $e$ in $C$ must be an edge from $b_i$ to some vertex $u'\in a$. In $G$, this corresponds to an edge in $T$ from some vertex $v'$ in component $i$ of $F'-a$ to $u'$. We add to $C'$ the entire path in $T$ from $v$ to $v'$ as well as the edge $v'u'$. We may assume that this path in $T$ does not have any vertices previously added to $C'$, for otherwise we already have a cycle $C''$ in $T$, contained in $C'$. Eventually the last edge of $C$ completes the cycle to $u_0$, and the corresponding edge added to $C'$ completes a cycle in $T$.

Next, we note that $L$ has at least $2$ components. If it had only one, then we could find a path between any two vertices of $F'$, which is impossible since $F'$ is disconnected (as it is a spanning subgraph of the disconnected graph $F$). Thus, the number of edges owned by $a$ which are in $T$ is $e(L) \leq |V(L)|-2 = |a|+c(F'-a)-2$.

Recall that $F'+e$ is a copy of $F$ in which $a$ is the set of newly activated vertices. So, if $e$ has at least one end in $a$, then $a$ owns at least $m_F(a) - 1$ edges under the old definition of ownership, and $c(F'-a) = c(F-a)$. On the other hand, if $e$ has neither end in $a$, then $a$ owns at least $m_F(a)$ edges under the old definition of ownership, and $c(F'-a) \leq c(F-a) + 1$. In either case, the number of edges owned by $a$ which are \emph{not} in $T$ is at least $m_F(a)-1 - (|a|+c(F-a) -2)$, as desired.
\end{proof}
For each $a\in \mathcal A$, let $m^*(a)$ denote the number of edges of $G$ owned by $a$ under the modified ownership. Combining claims \ref{claim:ownedinT} and \ref{claim:ownedoutsideT}, we have $m^*(a) \geq m_F(a) + 1 - c(F-a) - \mathbbm1_{r\in a}$ for each $a\in \mathcal A$. We need the following claim about $F$ to bound this quantity.
\begin{claim}\label{claim:Ffact} For any $\emptyset\subsetneq S \subseteq V(F)$, we have $m_F(S) +1 - c(F-S) \geq 15|S|/7$.\end{claim}
\begin{proof} Let $\emptyset\subsetneq S_*\subseteq V(F)$ be the subset minimizing the quantity $q(S_*) = (m_F(S_*)+1-c(F-S_*))/|S_*|$. We have $q(S_*) \leq 15/7$ by considering the choice $S_* = V(H)$.

First, let $T_1\subseteq S_*$ consist of vertices of $S_*$ which are part of the copy of $K_{100}$ inside $F$. Clearly $T_1 \neq S_*$ since $q(T_1) > 15/7$, but then it is easy to see that $q(S_*\setminus T_1) \geq q(S_*)$. Thus we may assume $T_1$ is empty.

Next, let $T_2 \subseteq S_*$ consist of vertices of $S_*$ which are part of the copy of $K_7$ inside $F$. Since $T_1$ is empty and $T_2$ has no edges to $S_* \sm T_2$, we see that $T_2$ contributes at least $6 |T_2| - \binom{|T_2|}{2} \geq 3|T_2|$ edges to $m_F(S_*)$, and if $T_2$ contains both vertices of the copy of $K_7$ which have edges to the copy of $K_{100}$ in $F$, then the number of edges is at least $3|T_2|+2$. Also, we have $c(F-S_*) \leq c(F-(S_*\sm T_2)) + 1$, and $c(F-S_*) = c(F-(S_*\sm T_2))$ unless $T_2$ contains both vertices of the copy of $K_7$ which have edges to the copy of $K_{100}$ in $F$. Thus overall, having any vertices in $T_2$ brings $q(S_*)$ closer to $3$, while $q(S_*) \leq 15/7$; this means we may assume $T_2$ is empty.

Therefore, $S_* \subseteq V(H)$. Exhaustive casework on $|S_*| \in \{1, \ldots, 7\}$ shows that we always have $q(S_*) \geq 15/7$, completing the proof. 
\end{proof}

Using \cref{claim:Ffact}, we thus have $m^*(a) \geq 15|a|/7 - \mathbbm 1_{r\in a}$. Recall that each edge of $G$ is owned by at most one part of the activation partition (see \cref{prop:ownership}); this is still true in the modified definition of ownership since we only reassigned some edges' ownership. Hence,
$$e(G) \geq \sum_{a\in \mathcal A} m^*(a) \geq \sum_{a\in \mathcal A} \left(15|a|/7 - \mathbbm 1_{r\in a}\right) = 15n/7 - 1,$$ where in the last step we used the fact that the parts $a\in \mathcal A$ partition the $n$ vertices of $G$. Since $G\in\uwSAT(n,F)$, we conclude that $w_F \geq 15/7$, completing the proof of \cref{prop:Fgraph}.
\end{proof}

\subsection{Other open questions}
\cref{thm:main} raises the natural question of whether weak saturation limits must be rational. We conjecture that the answer is yes.

\begin{conj}
For any graph $F$, $w_F$ is rational.
\end{conj} 

This would also follow from the stronger statement that once $n$ is large enough, minimum weakly $F$-saturated graphs on $n$ vertices can be formed by copying a fixed subgraph many times. Note that \cref{thm:sparse} implies that if $w_F < \delta_F/2$, then $w_F$ is rational; however, there could be irrational values of $w_F$ between $\delta_F/2$ and $\delta_F-1$. 

We close by discussing an analog of weak saturation limits for hypergraphs. Recently, Shapira and Tyomkyn \cite{ST23} proved that for any $r$-uniform hypergraph $H$, the limit $C_H:= \lim_{n\to\infty} \wsat(n, H) / n^{s-1}$ exists, where $s = s(H)$ is the minimum size of a vertex subset contained in exactly one edge of $H$. Another proof was later provided by Terekhov~\cite{t25}. It would be interesting to characterize the possible values of $C_H$ for any given uniformity $r \ge 3$.
\\\\
\textbf{Acknowledgments.} We extend our thanks to Yuval Wigderson, Hung-Hsun Hans Yu and Maksim Zhukovskii for helpful comments and conversations, and to Nikolai Terekhov for pointing out an error in an early version of this manuscript. We also thank anonymous referees for detailed feedback on our preprint. The first author was partially funded by the Georgia Tech ARC-ACO fellowship.

\end{document}